\documentclass[11pt, reqno,amsmath,amsthm,amssymb,amscd]{amsart}
\usepackage{mathrsfs,amssymb, amscd,amsmath,amsthm}
\usepackage[enableskew,vcentermath]{youngtab}
\usepackage{multicol}\multicolsep=0pt

\usepackage[final,backref=page]{hyperref}
\renewcommand*{\backref}[1]{}

\usepackage{amsthm}
\usepackage{amssymb}
\usepackage{amsmath}

\def\SUM#1#2{{\mbox{$\sum\limits_{#1}^{#2}$}}}

 \providecommand{\og}{``}
\providecommand{\fg}{''} \providecommand{\smfandname}{and}

\def\1{\hbox{1\kern-.35em\hbox{1}}}

\newtheorem{theorem}{Theorem}[section]
\newtheorem*{theorem*}{Theorem}
\newtheorem{lemma}[theorem]{Lemma}
\newtheorem{proposition}[theorem]{Proposition}
\newtheorem*{proposition*}{Proposition}

\newtheorem{coro}[theorem]{Corollary}

\theoremstyle{definition}
\newtheorem{definition}[theorem]{Definition}
\newtheorem{example}[theorem]{Example}

\theoremstyle{remark}
\newtheorem{remark}[theorem]{Remark}

\numberwithin{equation}{section}

\newcommand{\bea}{\begin{eqnarray}}
\newcommand{\eea}{\end{eqnarray}}
\newcommand{\be}{\begin{eqnarray*}}
\newcommand{\ee}{\end{eqnarray*}}

\def\t{{\mathfrak t}}
\def\s{{\mathfrak s}}
\newcommand{\Z}{{\mathbb Z}}

\newcommand{\C}{{\mathbb C}}

\newcommand{\mfg}{{\mathfrak g}}
\newcommand{\fh}{{\mathfrak h}}

\newcommand{\ft}{{\mathfrak t}}

\newcommand{\Hom}{{\rm Hom}}

\def\SUBS#1{}

\def\LE{_{\rm L}}
\def\RI{_{\rm R}}
\def\BO{_{\rm B}}

\def\a{\alpha}

\def\b{\beta}
\def\d{\delta}

\def\res{\text{res}}

\def\l{\lambda}
\def\L{\Lambda}
\def\si{\sigma}

\def\sc{\scriptstyle}
\def\ssc{\scriptscriptstyle}
\def\dis{\displaystyle}

\def\n{\mathfrak n}
\def\wt{\widetilde}

\def\Std{\mathcal T^{s}}

\def\N{\mathbb{N}}

\def\Z{\mathbb{Z}}

\def\C{\mathbb{C}}

\def\es{\varepsilon}

\def\Hom{{\rm Hom}}

\def\s{\mathfrak s}

\def\OTIMES{{{\sc\!}\otimes{\sc\!}}}
\def\a{\alpha}\def\b{\beta}\def\L{\Lambda}
\def\equa#1#2{
\begin{equation}\label{#1}#2\end{equation}}
\def\equan#1#2{$$#2$$}

\def\B{\mathscr B}

\def\FF{\frak F}
\def\GG{\frak G}
\def\FF{\frak F}
\def\DBr{{\mathscr  B_{2,r,t}}}
\def\OPL#1{\raisebox{-8pt}{$\stackrel{\mbox{\Large$\oplus$}}{\sc #1}$}}

\def\wt{{\rm wt}_\l}
\def\bii{{\vec{i}}}
\def\bij{{\vec{j}}}
\def\boldf{\textbf{\textit{f}}}
\def\boldg{\textbf{\textit{g}}}

\numberwithin{equation}{section}
\title[Singular vectors of (mixed) tensor products  ] {Highest weight vectors of mixed tensor products\\ of general linear Lie superalgebras
}

\author{Hebing Rui  {\normalfont \smfandname} Yucai \vspace*{-8pt}Su}
\address{H.~Rui: Department of Mathematics, Shanghai Key Laboratory of PMMP,  East China Normal
University, Shanghai, 200241, China} \email{hbrui@math.ecnu.edu.\vspace*{-5pt}cn}
\address{Y.~Su: Department of Mathematics, Tongji University,  Shanghai, 200092, China} \email{ycsu@tongji.edu.cn}
\thanks{Supported by NSFC (grant no.~11025104, 11371278), Shanghai Municipal Science and Technology Commission
~11XD1402200, 12XD1405000, and Fundamental Research Funds for the Central Universities of China.}
\begin{document}
\baselineskip16pt
\begin{abstract} In this paper, a notion of
 cyclotomic  (or level $k$)  walled Brauer algebras  $\mathscr B_{k, r, t}$
 is introduced for arbitrary positive integer $k$. It is proven that  $\mathscr B_{k, r, t}$   is  free over a commutative ring with  rank
$k^{r+t}(r+t)!$ if and only if  it is admissible. Using  super Schur-Weyl duality between general linear Lie
superalgebras $\mathfrak{gl}_{m|n}$ and  $\DBr$,  we give a classification of highest weight vectors of
$\mathfrak{gl}_{m|n}$-modules $M_{pq}^{rt}$,
the tensor products of  Kac-modules with mixed tensor products
of the natural module and its dual. This enables us to
establish an explicit relationship between  $\mathfrak{gl}_{m|n}$-Kac-modules
and  right cell  (or  standard) $\DBr$-modules   over $\mathbb C$.
Further, we find an explicit relationship between indecomposable tilting $\mathfrak{gl}_{m|n}$-modules appearing in
 $M_{pq}^{rt}$, and principal indecomposable right  $\DBr$-modules   via the notion of Kleshchev bipartitions.  As an application,
 decomposition numbers of $\DBr$ arising from super Schur-Weyl duality  are determined.\end{abstract}

\sloppy \maketitle

\section{Introduction}
Motivated by  Brundan-Stroppel's work on higher super Schur-Weyl duality in \cite{BS4}, we introduced
affine walled Brauer algebras  $\mathscr B^{\rm aff}_{r, t}$ in ~\cite{RSu} so as to establish higher super Schur-Weyl duality on
the tensor product $M_{pq}^{rt}$ of a Kac-module with a mixed tensor product
of the natural module and its dual for general linear Lie superalgebra $\mathfrak{gl}_{m|n}$ over $\mathbb C$ under the assumption
$r+t\le \min\{m, n\}$.\footnote{After we finished \cite{RSu}, Professor Stroppel informed us that  Sartori defined affine walled algebras via affine walled Brauer category, independently in \cite{Sa}.}
One of purposes of this paper is to generalize  super Schur-Weyl duality to the case $r+t>\min\{
m, n\}$. For this aim, we need to establish a bijective map from  a level two walled Brauer algebra $\DBr$ appearing in \cite{RSu} to a level two degenerate Hecke algebra $\mathscr H_{2, r+t}$.
This can be done by showing that  the dimension of $\DBr$ is $2^{r+t} (r+t)!$ over $\mathbb C$. We
consider this problem in a general setting by introducing
 a cyclotomic  (or level $k$)  walled Brauer algebra  $\mathscr B_{k, r, t}$
for arbitrary  $k\in \mathbb Z^{>0}$. By employing a totally new method, which is independent of seminormal forms of $\mathscr B_{k, r, t}$, we prove that $\mathscr B_{k, r, t}$   is  free over a commutative ring $R$ with  rank
$k^{r+t}(r+t)!$ if and only if  it is admissible in the sense of Definition~\ref{condi-k}.
It is  expected that  $\mathscr B_{k, r, t}$ can be used to study the problem on a classification of  finite dimensional simple
$\mathscr B^{\rm aff}_{r, t}$-modules over an algebraically closed field. Details will be given elsewhere.

The establishment of the  higher super Schur-Weyl duality~\cite{RSu}  enables us to use the representation theory of $\DBr$ to  classify highest weight vectors of  $M_{pq}^{rt}$ (at this point, we would like to
mention that purely on the Lie superalgebra  side, it seems to be hard to construct highest weight vectors of a given module, which is
 an interesting problem on its own right). On the other hand,
a classification of highest weight vectors of  $M_{pq}^{rt}$  also enables us to relate the category of finite dimensional $\mathfrak{gl}_{m|n}$-modules
with that of $\mathscr B_{2, r, t}$, which in turn gives us  an efficient way to calculate decomposition numbers of $\mathscr B_{2, r, t}$
 (cf. \cite{RSong} for quantum walled Brauer algebras). This  is the main motivation of this paper.
We explain some details below.

It is proven in \cite{RSu} that  $\text{End}_{U(\mathfrak {gl}_{m|n}) }(M_{pq}^{rt})\cong \DBr$ if $r+t\le \min\{m, n\}$. Since there is a bijection between the dominant weights of $M_{pq}^{rt}$ and the poset  $\Lambda_{2, r, t}$ in \eqref{poset}, and  since $\DBr$ is a weakly cellular algebra over $\Lambda_{2, r, t}$ in the sense of \cite{GG}, it is very natural to ask the following problem: whether a $\mathbb C$-space of $\mathfrak {gl}_{m|n}$-highest weight vectors  of $M_{pq}^{rt}$ with a fixed highest weight
is isomorphic to a cell (or standard) module of $\DBr$.

We give an
affirmative  answer to the problem. In sharp contrast to the Lie algebra case, due to the existence of the parity of $\mathfrak {gl}_{m|n}$, the known weakly cellular  basis of $\mathscr B_{2, r,t}$ in \cite{RSu} can not be directly used
to establish a relationship between $\mathfrak {gl}_{m|n}$-highest weight vectors of $M_{pq}^{rt}$ and right cell modules of $\mathscr B_{2, r, t}$.   One has  to find
  new cellular bases of level two Hecke algebra $\mathscr H_{2,r}$ which are different from that in \cite{AMR}. These new cellular bases of $\mathscr H_{2, r}$,
  which  relate both  trivial and signed representations of  symmetric groups, are used to
  construct a  new weakly cellular basis of $\DBr$. Motivated by  explicit descriptions of  bases of  right cell modules for $\DBr$, we construct and classify  $\mathfrak{gl}_{m|n}$-highest weight
  vectors of $M_{pq}^{rt}$.  This leads to a $\DBr$-module isomorphism between each   $\mathbb C$-space of $\mathfrak {gl}_{m|n}$-highest weight vectors
 of $M_{pq}^{rt}$  with a fixed highest weight  and the corresponding  cell module of $\DBr$.
Based on the above, we are able to construct a suitable exact functor sending $\mathfrak {gl}_{m|n}$-Kac-modules   to
 right cell modules of $\DBr$. This functor also sends an indecomposable tilting module appearing in $M_{pq}^{rt}$ to a principal  indecomposable right $\DBr$-module indexed by a pair of  so-called Kleshchev bipartitions in the sense of \eqref{Kle}. It  gives us an efficient way to  calculate  decomposition numbers of $\DBr$ via Brundan-Stroppel's result~\cite{BS4} on the multiplicity of a Kac-module in an indecomposable tilting  module appearing in $M_{pq}^{rt}$.

We organize the paper as follows. In section~2, after recalling the definition of   $\mathscr B^{\rm aff}_{r, t}$ over a commutative ring $R$, we introduce  cyclotomic walled Brauer algebras  $\mathscr B_{k, r, t}:=\mathscr B^{\rm aff}_{r, t}/I$ for arbitrary  $k\in \mathbb Z^{>0}$,  where $I$ is the two-sided ideal of $\mathscr B^{\rm aff}_{r, t}$ generated by two cyclotomic polynomials $\boldf(x_1)$ and $\boldg(\bar x_1)$ of degree $k$, which satisfy \eqref{cycpolyf}--\eqref{cycpolyg}.
When $\mathscr B^{\rm aff}_{r, t}$ is admissible in the sense of Definition~\ref{condi-k}, we describe explicitly an $R$-basis of $I$. This enables us to prove that $\mathscr B_{k, r, t}$ is free over $R$ with rank $k^{r+t}(r+t)!$
if and only  if it is admissible.
In section~3, we construct  cellular bases of $\mathscr H_{2, r}$ and use them  to construct  a weakly cellular basis of $\DBr$. In section~4, higher super Schur-Weyl dualities in \cite{RSu} are generalized to the case $r+t>\min\{m, n\}$.  In sections~5--6,
we classify   highest weight  vectors of $M_{pq}^{r0}$ and  $M_{pq}^{rt}$. Based on this, we
establish an explicit relationship between indecomposable tilting (resp. Kac) modules for $\mathfrak {gl}_{m|n}$ and
principal indecomposable (resp. cell) right $\DBr$-modules  via a suitable exact functor. This  gives us an efficient way to
  calculate decomposition numbers of $\mathscr B_{2, r, t}$ arising from the super Schur-Weyl duality in \cite{RSu}.

\section{Affine walled Brauer algebras and its cyclotomic quotients}

Throughout, we assume that $R$ is a commutative ring containing $\mathbf{\Omega}=\{\omega_a\mid a\in \mathbb N\}$ and identity $1$. In this section,
we introduce a level $k$ walled Brauer algebra $\mathscr B_{k, r, t}$ and  prove that
$\mathscr B_{k, r, t}$ is free over $R$ with  rank $k^{r+t} (r+t)!$ if and only if $\mathscr B_{k, r, t}$ is admissible in the sense of Definition~\ref{condi-k}.
First, we briefly recall
the definition of walled Brauer algebras.

 Fix  $r,t\in\Z^{>0}$. A {\it walled  $(r, t)$-Brauer diagram} (or simply, a {\it walled Brauer diagram}) is a diagram with $(r\!+\!t)$ vertices on  top and bottom rows,
and vertices on both rows are labeled  from left to right
by $r, ...,2, 1, \bar 1, \bar 2, ..., \bar t$, such that
every
$i\!\in\!\{r,...,2,1\}$ (resp., $\bar i\!\in $ $\{\bar 1, \bar 2,..., \bar t\}$) on each
row is
connected to a unique
$\bar j$ (resp., $j$) on the same row or a unique
$j$ (resp., $\bar j$) on the other row. Thus there are  four  types of pairs $[i,j],$ $[i,\bar j]$, $[\bar i,j]$ and $[\bar i,\bar j]$.
The pairs $[i, j]$ and $[\bar i, \bar j]$ are
{\it vertical edges}, and
$[\bar i, j]$ and $[i, \bar j]$ are
{\it horizontal edges}.

The product of two  walled Brauer diagrams $D_1$ and $D_2$ can be defined via concatenation.
Putting $D_1$ above $D_2$ and connecting  each
 vertex on the bottom row of $D_1$ to the corresponding vertex on the top row of $D_2$ yields a diagram $D_1\circ D_2$, called the \textit{concatenation} of
 $D_1$ and $D_2$.
 Removing all circles of $D_1\circ D_2$ yields  a unique walled Brauer diagram,  denoted  $D_3$. Let $n$ be the number of circles appearing in
 $D_1\circ D_2$. Then the {\it product} $D_1 D_2$ is defined to be $\omega_0^{n} D_3$,
 where $\omega_0$ is a fixed element in $R$.
The {\textit{walled Brauer algebra}}~\cite{Koi, Tur, N}  $\mathscr B_{r, t}:=\mathscr B_{r, t}(\omega_0)$ with  defining parameter $\omega_0$
is the
associative $R$-algebra spanned by all walled Brauer diagrams with product defined in this way.

Let $\mathfrak S_r$ (resp. $ {\bar{\mathfrak S}}_{t}$) be the symmetric group in
  $r$ (resp. $t$)  letters $r,...,2,1$ (resp. $\bar 1, \bar 2,..., \bar t$).
It is known that $\mathscr B_{r, t}$ contains two  subalgebras which are   isomorphic to the group algebras of $\mathfrak S_r$ and ${\bar{\mathfrak S}}_{t}$, respectively.
More explicitly, the walled Brauer diagram $s_i$ whose edges are of forms $[k, k]$ and $[\bar k, \bar k]$ except two vertical edges $[i, i+1]$ and $[i+1, i]$
can be identified with the basic transposition $(i, i+1)\in \mathfrak S_r$, which switches $i$ and $i+1$ and fixes others.  Similarly, there is a  walled Brauer
diagram $\bar s_j$  corresponding to  $(\bar j, \overline{j+1})\in {\bar{\mathfrak S}}_{t}$.  Let $e_1$ be the walled Brauer diagram whose edges are
of forms $[k, k]$ and $[\bar k, \bar k]$ except two horizontal edges $[1,\bar 1]$ on the top and bottom rows. Then $\mathscr B_{r, t}$ is the  $R$-algebra \cite{N} generated by
$e_1$, $s_i$, $\bar s_j$ for $1\le i\le r-1$, $1\le j\le t-1$
such that $s_i$'s, $\bar s_j$'s are distinguished  generators of $\mathfrak S_r\times{\bar{\mathfrak S}}_{t}$ and
\begin{equation}\label{walled}\begin{array}{lll}
e_1^2=\omega_0 e_1,\ \ \ \ e_1s_1e_1=e_1=e_1\bar s_1e_1,&& s_ie_1=e_1s_i,\ \ \bar s_je_1=e_1\bar s_j\ \ (i,j\ne1),\\
e_1s_1\bar s_{1} e_1s_{1} = e_1s_1\bar s_1 e_1\bar s_{1},&&
s_1e_1s_1\bar s_1 e_1 = \bar s_1 e_1s_1 \bar s_1 e_1.\end{array}\end{equation}

 Let  $\mathscr H_n^{\rm {aff}}$ be the {\it degenerate affine Hecke algebra}~\cite{Dri}.
As a free $R$-module, it is the tensor product $R[y_1, y_2, \cdots, y_n]\otimes R\mathfrak S_n$ of
a polynomial algebra with the group algebra of  $\mathfrak S_n$.  The multiplication is defined so that $ R[y_1, y_2, \cdots, y_n]\equiv R[y_1, y_2, \cdots, y_n]\otimes 1$
and $R\mathfrak S_n\equiv 1\otimes R\mathfrak S_n$  are subalgebras and $s_i y_j=y_js_i$ if $j\neq i, i+1$ and $s_iy_i=y_{i+1}s_i-1$, $1\le i\le n-1$.

Recall that $R$ contains $1$ and $\mathbf \Omega=\{\omega_a\in R\mid a\in \mathbb N\}$.  The  {\textit {affine walled Brauer algebra}}
$\mathscr B_{r, t}^{\rm {aff}}(\mathbf \Omega)$ (which is  $\widehat{\mathscr B_{r,t}}$  in \cite[\S 4]{RSu})
 with respect to the defining parameters $\omega_a$'s  have been defined via generators and $26$  defining relations~\cite[Definition~2.7]{RSu}.
It follows from  \cite[Theorem~4.15]{RSu} that  $\mathscr B_{r, t}^{\rm {aff}}(\mathbf \Omega)$ can be also defined in a simpler way  as follows:
it is an associative $R$-algebra generated by  $e_1,\,x_1,\,\bar x_1,\,s_i,\,\bar s_j$ for $1\le i\le r-1,$ $1\le j\le t-1$, such that $e_1$, $s_i$'s, $\bar s_j$'s are generators of
 $\mathscr B_{r, t}$ with defining parameter $\omega_0$, and as a free $R$-module,
$$\mathscr B_{r, t}^{\rm {aff}}(\mathbf \Omega)=R[\mathbf {x}_r]\otimes \mathscr B_{r, t}\otimes R[\bar {\mathbf {x}}_t],$$ the tensor product of the walled Brauer algebra $\mathscr B_{r, t}$ with
 two polynomial algebras $$\text{ $R[\mathbf{x}_r]:=R[x_1, x_2, \cdots, x_r]$, \ and \ $R[\bar{ \mathbf{x}}_t]:=R[\bar x_1, \bar x_2, \cdots, \bar x_t]$.}$$
  Multiplication is defined as  $ R[\mathbf{x}_r]\equiv R[\mathbf{x}_r]\otimes 1\otimes  1$, $R[\bar{ \mathbf{x}}_r]\equiv 1\otimes 1\otimes R[\bar{ \mathbf{x}}_r]$,
and $\mathscr B_{r, t}\equiv $ $1\otimes \mathscr B_{r, t}\otimes 1$, and $R[\mathbf{x}_r]\otimes R\mathfrak S_r\otimes 1\cong \mathscr H_r^{\rm{aff}}\otimes 1$,
$1\otimes R{\bar{\mathfrak S}}_{t}\otimes R[\bar{ \mathbf{x}}_r]\cong 1\otimes \mathscr H_t^{\rm{aff}}$
 and
\begin{eqnarray}
\label{it1}&\!\!\!\!\!\!\!\!\!\!\!\!\!\!\!\!&
e_1(x_1+\bar x_1)=(x_1+\bar x_1) e_1=0,\ \ s_1e_1 s_1x_1=x_1s_1e_1s_1,\ \  \bar s_1 e_1 \bar s_1\bar x_1=\bar x_1\bar s_1e_1\bar s_1,\\
\label{it2}&\!\!\!\!\!\!\!\!\!\!\!\!\!\!\!\!&
s_i\bar x_1=\bar x_1 s_i,\ \ \bar s_i x_1= x_1 \bar s_i,\ \  x_1 (e_1+\bar x_1)= (e_1+\bar x_1)x_1,\\
\label{it3}&\!\!\!\!\!\!\!\!\!\!\!\!\!\!\!\!&
e_1x_1^ke_1=\omega_k e_1,\ \  e_1\bar x_1^ke_1=\bar \omega_k e_1, \ \ \forall k\in \mathbb Z^{\ge0},
\end{eqnarray}
where $\bar \omega_a$'s are  determined by \cite[Corollary~4.3]{RSu}. If  $\bar \omega_a$'s  do not satisfy \cite[Corollary~4.3]{RSu}, then $e_1=0$ and
$\mathscr B_{r, t}^{\rm {aff}}(\mathbf \Omega)$ turns out to be $\mathscr H_r^{\rm aff} \otimes \mathscr H_t^{\rm aff}$ if $R$ is a field.

 We remark that   the isomorphism $R[\mathbf{x}_r]{\!}\otimes{\!} R\mathfrak S_r{\!}\otimes {\!}1{\!}\cong {\!}\mathscr H_r^{\rm{aff}}{\!}\otimes {\!}1$
sends $s_i$'s (resp. $x_1$) to
 $s_i$'s (resp. $-y_1$), and the isomorphism  $1\otimes R{\bar{\mathfrak S}}_{t}\otimes R[\bar{ \mathbf{x}}_r]\cong 1\otimes \mathscr H_t^{\rm{aff}}$
sends $\bar s_j$'s (resp. $\bar x_1$) to $s_j$'s (resp.
$-y_1$).  So,  $x_{i+1}=s_ix_is_i-s_i$ and $\bar x_{j+1}=\bar s_j\bar x_j\bar s_j-\bar s_j$ and
$y_{i+1}=s_i y_is_i+s_i$ if all of them make sense.

For the simplification of notation,  we denote   $\mathscr B_{r, t}^{\rm {aff}}(\mathbf \Omega)$  by  $\mathscr B_{r, t}^{\rm {aff}}$.
Fix $u_1, u_2, \cdots, u_k\in R$ for some $k\in \mathbb Z^{>0}$. Let $\boldf(x_1)\in \mathscr B_{r, t}^{\rm aff} $ be  such that
\begin{equation} \label{cycpolyf} \boldf(x_1)=\mbox{$\prod\limits_{i=1}^k$} (x_1-u_i).\end{equation}
 By  \cite[Lemma~4.2]{RSu} (or using  \eqref{it1}--\eqref{it2}), there is a monic  polynomial $\boldg(\bar x_1)\in R[\bar x_1]$ with degree $k$ such that
 \begin{equation} \label{efg} e_1 \boldf(x_1)=(-1)^k e_1 \boldg(\bar x_1).\end{equation} If  $R$ is an algebraically closed field, then there are $\bar u_1, \bar u_2, \cdots, \bar u_k\in R$ such that
  \begin{equation}\label{cycpolyg}
 \boldg(\bar x_1)=\mbox{$\prod\limits_{i=1}^k$} (\bar x_1-\bar u_i).
 \end{equation}

\begin{definition} \label{cycwb} Let $R$ be a commutative ring containing $1$, $\mathbf \Omega=\{\omega_a\in R\mid a\in \mathbb N\}$, and $u_i, \bar u_i$, $1\le i\le k$.
 The cyclotomic (or level $k$)  walled Brauer algebra $\mathscr B_{k, r, t}$
 is the quotient algebra $\mathscr B_{r, t}^{\rm aff} /I$, where $I$ is the two-sided ideal
of $\mathscr B_{r, t}^{\rm aff}$ generated by $\boldf(x_1)$ and $\boldg(\bar x_1)$ satisfying  \eqref{cycpolyf}--\eqref{cycpolyg}.
\end{definition}

If $k=1$, then $\mathscr B_{k, r, t}$ is  $\mathscr B_{r, t}$ with defining parameter $\omega_0$. For some special $u_i, \bar u_i$, $i=1, 2$, $\mathscr B_{2, r, t}$
is the level two walled Brauer algebras arising from super Schur-Weyl duality in
\cite{RSu}.

\begin{lemma}\label{lemm-level-k-walled} Let $\boldf(x_1)$ be given in \eqref{cycpolyf}. Write   $\boldf(x_1)=x_1^k+\sum_{i=1}^{k} a_i x_1^{k-i} $. Then  $e_1$ is an $R$-torsion
element of $ \mathscr B_{k, r, t}$  unless
\begin{equation}\label{condi-k1}
 \omega_\ell=-(a_1\omega_{\ell-1}+\cdots a_k\omega_{\ell-k})\mbox{ \ for all \ }\ell\ge k.
\end{equation}
\end{lemma}
\begin{proof} Let  $b_\ell=\omega_\ell+a_1\omega_{\ell-1}+\cdots+ a_k\omega_{\ell-k}\in R$. By \eqref{it3}, $b_\ell e_1=e_1\boldf(x_1)x_1^{\ell-k}e_1$ in $\mathscr B_{r, t}^{\rm aff}$ and $b_\ell e_1=0$  in  ${\mathscr B}_{k,r,t}$.
Thus, $e_1$ is  an $R$-torsion element if  $b_\ell\neq 0$ for some $\ell\ge k$.
\end{proof}

\begin{definition}\label{condi-k}  The algebras $\mathscr B_{r, t}^{\rm aff}$ and   ${\mathscr B}_{k,r,t}$ are called  \textit{admissible}  if \eqref{condi-k1} holds.\end{definition}

\begin{lemma}\label{zero-1} Assume  $\boldf(x_1), \boldg(\bar x_1)\in \mathscr B_{r,t}^{\rm aff}$ satisfying  \eqref{cycpolyf}--\eqref{cycpolyg} . If  $\mathscr B_{r, t}^{\rm aff}$ is admissible, then
\begin{enumerate} \item $e_1 \boldf(x_1) x_1^a e_1=0$ for all $a\in \mathbb N$.
\item  \label{zero-11} $e_1 \boldg(\bar x_1)\bar x_1^a e_1=0$ for all $a\in \mathbb N$.\end{enumerate}
 \end{lemma}
\begin{proof} (1) trivial since  $\mathscr B_{r, t}^{\rm aff}$ is admissible. It is proven in \cite{RSu} that  there is an $R$-linear anti-involution $\sigma$ on $\mathscr B_{r, t}^{\rm aff}$,
which fixes all generators of $\mathscr B_{r, t}^{\rm aff}$.
Applying  $\sigma$ on  \cite[Lemma~4.2]{RSu}  yields  $$ \bar x_1^ke_1=\mbox{$\sum\limits_{i=0}^k$} a_{k, i} x_1^i e_1, \text{  for some $a_{k, i}\in R$.}$$
So, (2) follows from \eqref{efg} and (1), immediately. \end{proof}

 Denote $s_{i, j}=s_is_{i+1}\cdots s_{j-1}$ if $i<j$, and $1$ if $i=j$, and $s_{i-1}s_{i-2}\cdots s_{j}$ if $i>j$.  Denote $\bar s_{i, j}\in \bar {\mathfrak S}_{ t}$ similarly. Let $e_{i,j}$ be the walled Brauer diagram such that each vertical edge of $e_{i, j}$ is of form $[k, k]$ or $[\bar k, \bar k]$ and the  horizontal edges
 on the top and bottom rows of $e_{i, j}$ are $[i, \bar j]$. Then
 \begin{equation}\label{e-ij}e_{i, j}= \bar s_{j, 1} s_{i, 1} e_1 s_{1, i} \bar s_{1, j}
\mbox{ \ for  $i, j$ with  $1\le i\le r $ and $1\le j\le t$.}\end{equation}
 For each nonnegative integer $f\le \min\{r, t\}$, let  \equa{e-f==}{\mbox{$e^f=e_1 e_2\cdots e_f$ for $f>0$ and $e^0=1$, where $e_i=e_{i,i}$.}} Set
 \begin{equation}\label{rcs11}  \mathscr{D}_{r,t}^f=\{ s_{f,i_f} \bar s_{f, j_f} \cdots
s_{1,i_1}\bar s_{1,{j_1}}\,|\,
 1 {\sc\!}\le{\sc\!} i_1{\sc\!}<{\sc\!} \cdots{\sc\!} <{\sc\!}i_f\le r,\,k {\sc\!}\le{\sc\!} {j_k}\}.\end{equation}

 \begin{definition}\label{def-mo} For $\alpha=(\alpha_1, \cdots, \alpha_r)\in  \mathbb N^r$  and $\beta=(\beta_1, \cdots, \beta_t)\in \mathbb N^t$,
 let  $x^\alpha= \prod_{i=1}^r x_i^{\alpha_i} $, ${\bar x}^\beta=\prod_{j=1}^t \bar x_j^{\beta_j}$.  Let $\mathcal M$ be a subset of $\mathscr B_{r,t}^{\rm aff}$ given by
 \begin{equation} \label{regs}  \mathcal M=
 \stackrel{\sc \min\{m, n\}}{\underset{\sc f=0}{\dis\mbox{\Large$\cup$}}}\{x^{\alpha}  c^{-1} e^f w d  \bar x^{\beta}  \mid (\alpha, \beta)\in \mathbb N^r \times  \mathbb N^t, c, d\in \mathscr D_{r, t}^f,
  w\in \mathfrak S_{r-f}\times \bar{\mathfrak S}_{{t-f}}\}.\end{equation}
 Elements of $\mathcal M$  are called   \textit{ regular monomials } of $\mathscr B_{r,t}^{\rm aff}$.\end{definition}

\begin{theorem}\label{monothm}\cite[Theorem~4.15]{RSu} The affine walled Brauer algebra  $\mathscr B_{r, t}^{\rm aff}$ is free over $R$ with $\mathcal M$ as its  $R$-basis. \end{theorem}

We consider $\mathscr B_{r,t}^{\text{aff}}$ as a filtrated $R$-algebra  as follows. Let
$$ \text{ $\text{deg}{\sc\,} s_i=$ $\text{deg}{\sc\,} \bar s_j= \text{deg} e_1=0$ and
$\text{deg}{\sc\,}{x_k}=\text{deg}{\sc\,} \bar x_\ell=1$}$$ for all possible $ i, j, k, \ell$'s. Let $(\mathscr B_{r,t}^{\text{aff}})^{(k)}$ be the  $R$-submodule
 spanned by regular monomials with degrees  less than or equal to $k$ for  $k\in{\mathbb{Z}}^{\ge0}$. Then we have the following filtration
  \begin{equation}\label{filtr}
\mathscr B_{r,t}^{\text{aff}}\supset\cdots\supset (\mathscr B_{r,t}^{\text{aff}})^{(1)}\supset(\mathscr B_{r,t}^{\text{aff}})^{(0)}\supset (\mathscr B_{r,t}^{\text{aff}})^{(-1)}=0.\end{equation}
Let ${\rm gr} ( \mathscr B_{r,t}^{\text{aff}})=\oplus_{i\ge0}( \mathscr B_{r,t}^{\text{aff}})^{[i]}$, where
$( \mathscr B_{r,t}^{\text{aff}})^{[i]}=( \mathscr B_{r,t}^{\text{aff}})^{(i)}/( \mathscr B_{r,t}^{\text{aff}})^{(i-1)}$. Then
  ${\rm gr} ( \mathscr B_{r,t}^{\text{aff}})$ is an  associated $\mathbb Z$-graded algebra.  We will use the same symbols to denote elements in
${\rm gr} ( \mathscr B_{r,t}^{\text{aff}})$.

\begin{lemma}  \label{x-p-iii} Let   $
x'_i=s_{i-1} x_{i-1}' s_{i-1}$, and $\bar x'_j=\bar s_{j-1}\bar x_{j-1} \bar s_{j-1}$ for $i, j\in \mathbb Z^{\ge 2}$ with $i\le r$ and $j\le t$, where  $x_1'=x_1$, and  $\bar x_1'=\bar x_1$. \begin{enumerate} \item
$ x_i=x'_i-L_i$,
where $L_i=\sum_{1\le j<i} (j, i)$ and $(j, i)$  is the  transposition in $\mathfrak S_r$  which switches $j, i$ and fixes others.
\item  $\bar x_i=\bar x'_i-\bar L_i$, where $ \bar L_i=\sum_{\bar 1\le \bar j<\bar i} (\bar j, \bar i)$ and  $(\bar j, \bar i)$ is the  transposition in
  $\bar {\mathfrak S}_{ t}$ which switches $\bar j, \bar i$ and fixes others.
  \item Any symmetric polynomial of $L_1, L_2, \cdots, L_r$ $($resp. $\bar L_1, \bar L_2, \cdots, \bar L_t\,)$ is a central element of $R\mathfrak S_r$ $($resp. $R\bar{\mathfrak S}_{\t}\,)$.\end{enumerate}
\end{lemma}
\begin{proof} (1)-(2) are trivial and (3) is a well-known result. \end{proof}

The elements $L_i$'s (resp. $\bar L_j$'s) are known as Jucys-Murphy elements of $R\mathfrak S_r$ (resp. $R\bar{\mathfrak S}_{\t}$).
Note that  $x_i x_j=x_jx_i$
and $\bar x_i\bar x_j=\bar x_j\bar x_i$ for all possible $i, j$. However, $x_i'$ and $x_j'$ (resp. $\bar x_i'$ and $\bar x_j'$)  do not commute each other.

Suppose  $0< f\le \min\{m, n\}$. Denote \equa{bii-bij}{\bii=(i_1,..,i_f),\ \ \bij=(j_1,...,j_f),\ \ \ e_{\bii,\bij}=e_{i_1,j_1}e_{i_2,j_2}\cdots e_{i_f,j_f},}
where $i_1, i_2, \cdots, i_f$ are distinct numbers in $\{1, 2, \cdots, r\}$, and  $j_1, j_2, \cdots, j_f$ are distinct numbers in $\{\bar 1, \bar 2, \cdots, \bar t\}$.
Then $e_{i_k, j_k}$'s commute each other.   If $f=0$, we set $\bii=\bij=\emptyset$ and $e_{\bii,\bij}=1$.

We always assume that $\mathfrak S_r$ (resp. $\bar {\mathfrak S}_t$) acts on the right of $\{r,...,2,1\}$ (resp. $\{\bar 1, \bar 2, ..., \bar t\}$).
 \begin{lemma}\label{lemm-----2}
Suppose $a \in\Z^{>0}$, $1\le i, \ell\le r$ and $1\le j\le t$.
\begin{enumerate}\item If  $w\in  \mathfrak S_r$, then
$w\boldf(x'_i){{w}}^{-1}=\boldf(x'_{(i) w^{-1}})$.
\item If $ w\in {\bar{\mathfrak S}}_t$, then
$ w\boldg(\bar x'_{\bar j}) w^{-1}=\boldg(\bar x'_{ (\bar j)w^{-1}})$.
\item  $x'^a_i\boldf(x'_\ell)=\boldf(x'_\ell)x'^a_i+v$, where $v\in\sum_{b<a}\sum_{h, h_1=1}^{\max\{i, \ell\}} \boldf(x'_h)x'^b_{h_1}R\mathfrak S_r$.
\item  $\bar x'^a_j\boldf(x'_i)=\boldf(x'_i)\bar x'^a_j+v$, where $v\in \sum_{b_1+b_2<a,\,c_1+c_2\le1} \epsilon \bar x'^{b_1}_j e^{c_1}_{ij}\boldf(x'_i)e_{ij}^{c_2}\bar x'^{b_2}_j$ for some non-negative integers  $b_1,b_2,c_1,c_2$ and $\epsilon=\pm 1$. \end{enumerate}
\end{lemma}
\begin{proof} (1)-(2) are trivial.   Since $x_2=x_2'-s_1$ and $x_2x_1=x_1x_2$, \begin{equation}\label{key-1}
x_2' \boldf(x_1)=\boldf(x_1)(x_2'-s_1)+\boldf(x_2') s_1. \end{equation}
Applying the conjugate of  $s_{i, 2}$ on (\ref{key-1}) yields  (3) for $a=1$ and $\ell=1$.
If $\ell>1$, then
   $x_i' \boldf(x_\ell')=x_i's_{\ell-1} \boldf(x_{\ell-1}') s_{\ell-1}= s_{\ell-1}x_{(i)s_{\ell-1}}' \boldf(x_{\ell-1}') s_{\ell-1}$.
   Thus, (3) follows from  inductive assumption on $\ell-1$ and (1) under the assumption $a=1$. The case $a>1$ follows by using the previous result on $a=1$, repeatedly.
Finally, (4) can be checked similarly by induction. We leave the details to the readers.
\end{proof}

\begin{proposition}~\label{2ij} Let  $J_L=\sum_{i=1}^t {\mathscr B}_{r, t}^{\rm aff}\, \boldg(\bar x'_j)$ and   $J_R=\sum_{i=1}^r \boldf(x_i') \mathscr B_{r, t}^{\rm aff}$.  Then
\begin{enumerate}\item $J_L$ is a right $R\mathfrak S_r\otimes  \mathscr H_{t}^{\rm aff}$-module;
\item $J_R$ is a left $\mathscr H^{\rm aff}_r\otimes R\bar {\mathfrak S}_t$-module;
\item if  $\mathscr B_{r, t}^{\rm {aff}}$ is admissible,  then $I=J_L+J_R$,
where $I$ is the two-sided ideal of ${\mathscr B}_{r, t}^{\rm aff}$  generated by
$\boldf(x_1)$ and $ \boldg(\bar x_1)$ satisfying \eqref{cycpolyf}--\eqref{cycpolyg}.\end{enumerate}
\end{proposition}

\begin{proof} Obviously, both $J_L$ and $J_R$ are   $ \mathfrak S_r\times\mathfrak S_{\bar t}$-bimodules. By Lemma~\ref{lemm-----2}\,(3),
 $x_1 J_R\subseteq J_R$.
Similarly,  $J_L\bar x_1\subseteq J_L$. This proves (1)--(2).  In order to prove (3), it suffices to verify that $J_L+J_R$ is a two-sided ideal of
$\mathscr B_{r, t}^{\rm aff} $. If so, since  $\{\boldf(x_1), \boldg(\bar x_1)\}\subset J_L+J_R$,    $I=J_L+J_R$, proving the result.

We claim that $e_1 (J_L+J_R)\subseteq J_L+J_R$ and $ (J_L+J_R) e_1 \subseteq J_L+J_R$. If so, by \eqref{it2},
 $(\bar x_1 +e_1) \boldf(x_1)=\boldf(x_1)(\bar x_1+e_1)$ and hence   $\bar x_1 \boldf(x_1)\in J_L+J_R$.
By (1)--(2),  $\bar x_1 \boldf(x_i')=s_{i, 1} \bar x_1 \boldf(x_1) s_{1, i}\in J_L+J_R$, and hence
 $\bar x_1(J_L+J_R)\subseteq J_L+J_R$. Similarly,  $(J_L+J_R) x_1\subseteq J_L+J_R$.
Thus the claim implies that     $J_L+J_R$ is a two-sided ideal of $\mathscr B_{r, t}^{\rm aff} $.

By symmetry, it remains to prove  $e_1(J_L+J_R)\subseteq J_L+J_R$. Obviously, it suffices to verify  \begin{equation}\label{ejr} e_1 J_R\subset J_L+J_R.\end{equation}
 By \eqref{it1}, $e_1 \boldf(x_i')=\boldf(x_i') e_1$ for $i\ge 2$. Let  ${\textbf{m}}$ be a regular monomial of $\mathscr B_{r, t}^{\rm aff}$ defined in~\eqref{regs}.
 Then $\textbf{m}
 =x^\a e_{\bii,\bij} w\bar x^\beta$ for some $w\in \mathfrak S_r\times \bar {\mathfrak S}_t$, $(\a, \b)\in \mathbb N^r\times \mathbb N^t$ and some $\bii, \bij$.
 Using induction on $|\a|$,  we want to prove \begin{equation}\label {condreg} e_1 \boldf(x_1) {\textbf{m}}\in J_L+J_R.\end{equation}
If so, then  $e_1 \boldf(x_1)\mathscr B_{r, t}^{\rm aff}\subset J_L+J_R$ and hence \eqref{ejr} follows.

\medskip

\textit{Case~1: $|\alpha|=0$.}

If $f=0$, then (\ref{condreg}) follows from (1) and  \eqref{efg}. Suppose   $1\le f\le \min\{r, t\} $. Since $\mathscr B_{r,t}^{\rm aff} $ is admissible,    $e_1 \boldf(x_1) {\textbf{m}}=0$ if $e_i$ is a factor
of  $e_{\bii,\bij}$. Assume that   $e_1$ is not a factor of $e_{\bii,\bij}$.
If there is an $l$ such that $i_l=p\neq 1$ and $j_l=1$, by (2),
$$e_1 \boldf(x_1) e_{p, 1}=s_{p, 2} e_1 \boldf(x_1) s_1 e_1 s_{1, p}=s_{p, 2} e_1s_1 \boldf(x_2') e_1 s_{1, p}=s_{p, 2}  \boldf(x_2')e_1 s_{1, p}\in J_R.$$
Suppose  $j_l\neq 1$ for all possible $l$.
If there is an $l$ such that $e_{i_l, j_l}=e_{1, p}$ for some $p\neq 1$, then we assume $i_1=1$ and $j_1=p$ without loss of any generality.
In this case, $$e_1 \boldf(x_1) e_{1, p}=(-1)^k \bar s_{p, 2} e_1 \boldg(\bar x_1) \bar s_1 e_{1} \bar s_{1, p}=(-1)^k \bar s_{p, 2} e_1 \boldg(\bar {x}_2') \bar s_{1,p}=(-1)^k \bar s_{p, 2} e_1 \bar s_{1, p} \boldg(\bar x_1).$$
Since $j_l\neq 1$ for $1\le l\le f$, by \cite[Lemma~4.7(2)]{RSu},  $\bar x_1 e_{i_l, j_l}=e_{i_l, j_l} \bar x_1$ and hence
  \begin{equation}\label{ge} \boldg(\bar x_1)\mbox{$ \prod\limits_{l=2}^f e_{i_l, j_l}=  \prod\limits_{l=2}^f $}e_{i_l, j_l} \boldg(\bar x_1)\in J_L.\end{equation}
Now,  \eqref{condreg} follows from (1).
Finally, if  $\{ i_l, j_l\}\cap \{1\}=\emptyset$ for all possible $l$, then   (\ref{condreg}) follows from (1) and the following fact
$$e_1 \boldf(x_1)\mbox{$ \prod\limits_{ l=1}^f e_{i_f, j_f}=\prod\limits_{ l=1}^f e_{i_f, j_f}e_1 \boldf(x_1) =(-1)^k \prod\limits_{ l=1}^f $}e_{i_f, j_f} e_1 \boldg(\bar x_1)\in J_L.$$

 \textit{Case~2: $|\alpha|>0$.}

 If $\alpha_i\neq 0$ for some $2\le i\le r$, then
 $e_1 x_i=x_i' e_1-e_1\sum_{j=1}^i (j, i)$ and $x_i \boldf(x_1)=\boldf(x_1)x_i$.
 Let $\textbf{m}'$ be obtained from $\textbf{m}$ by removing $x_i$. Then
 $$e_1 (1, i) \boldf(x_1) \textbf{m}'=e_1 \boldf(x_i') (1, i)   \textbf{m}'=\boldf(x_i') e_1 (1, i)   \textbf{m}' \in J_R.$$
Now, \eqref{condreg} follows from inductive assumption on $|\a|$.
If $\alpha_i=0$, $2\le i\le r$, then  $x^\a=x_1^{\a_1}$ with $\a_1>0$.
 Let $v=e_1\boldf(x_1) \textbf{m}$. If
 $j_\ell\neq 1$, $1\le \ell\le f$, then by \eqref{ge}, Lemma~\ref{lemm-----2} and inductive assumption,
$$\begin{aligned}
v&=e_1\boldf(x_1)x_1^{\a_1}e_{\bii,\bij} w\bar x^\b
= (-1)^k  e_1\boldg(\bar x_1)x_1^{\a_1}e_{\bii,\bij} w\bar x^\b
 \equiv (-1)^k  e_1x_1^{\a_1}\boldg(\bar x_1)e_{\bii,\bij} w\bar x^\b\\ &
= (-1)^k e_1x_1^{\a_1}e_{\bii,\bij}{\sc\,}\boldg(\bar x_{1}) w{\bar x}^\b \nonumber \in J_L w {\bar x}^\b\subset J_L+J_R,\\
\end{aligned}$$
where the ``\,$\equiv$\,'' is  modulo $J_L+J_R$.
Finally, if $j_\ell=1$ for some $\ell$,  without loss of any generality, we assume  $j_1=1$.  If $i_1=1$, by Lemma~\ref{zero-1},
$v=e_1\boldf(x_1)x_1^{\a_1}e_{1}e_{\bii',\bij'} w\bar x^\b=0$, where $\bii'=(i_2,...,i_f)$ and  $\bij'=(j_2,...,j_f)$.  Now, we assume $i_1\neq 1$.
 Then
$$\begin{aligned} v&= e_1\boldf(x_1)x_1^{\a_1}e_{i_1,1}e_{\bii',\bij'} w\bar x^\b
=e_1e_{i_1,1}\boldf(x_1)x_1^{\a_1}e_{\bii',\bij'} w\bar x^\b\\
&=e_1(1,i_1)\boldf(x_1)x_1^{\a_1}e_{\bii',\bij'} w\bar x^\b
= e_{1}\boldf(x_i') (1, i) x_1^{\a_1}e_{\bii',\bij'} w\bar x^\b,\\ &
=\boldf(x_i') e_1 (1, i) x_1^{\a_1}e_{\bii',\bij'} w\bar x^\b\in  J_R.\\
\end{aligned}$$
This completes the proof of \eqref{condreg}. \end{proof}

For $(\a, \b)\in \mathbb N^r\times \mathbb N^t$, denote
$\boldf(x')^\alpha=\boldf(x_1)^{\a_1}\cdots \boldf(x_r')^{\a_r}$ and $\boldg(\bar x')^\b=\boldg(\bar x_1)^{\b_1}\cdots \boldg(\bar x_t')^{\b_t}$.
Let $\mathbb N_k^r=\{\a\in \mathbb N^r\mid  \a_i\le k-1, 1\le i\le r\}$ and  $\mathbb N_k^t=\{\a\in\mathbb N^t\mid  \a_i\le k-1, 1\le i\le r\}$.

\begin{lemma}\label{mathcaln} The affine walled Brauer algebra $\mathscr B_{r,t}^{\rm aff}$ is a free $R$-module with $\mathcal N$ as its $R$-basis, where
 \begin{equation} \label{regs1} \begin{aligned}  \mathcal N= \stackrel{\sc \min\{m, n\}}{\underset{\sc f=0}{\dis\mbox{\Large$\cup$}}}
 & \{ \boldf(x')^\a {x}^{\gamma}  c^{-1} e^f w d \bar x^{\delta}\boldg(\bar x')^\b
 \mid (\alpha, \beta)\in\mathbb N^r \times  \mathbb N^t, (\gamma, \delta)\in\mathbb N_k^r \times  \mathbb N_k^t, \\[-8pt]
 & \ c, d\in \mathscr D_{r, t}^f,
  w\in \mathfrak S_{r-f}\times \bar{\mathfrak S}_{{t-f}}\}.\\ \end{aligned} \end{equation}\end{lemma}
\begin{proof} The result follows from  Theorem~\ref{monothm} since the transition matrix between $\mathcal N$ and $\mathcal M$ in \eqref{regs} is invertible.
\end{proof}

\begin{lemma}\label{I}Let $I$ be the two-sided ideal of  $\mathscr B_{r,t}^{\rm aff}$ generated by
$\boldf(x_1)$ and $\boldg(\bar x_1)$ satisfying \eqref{cycpolyf}--\eqref{cycpolyg}. If  $\mathscr B_{r,t}^{\rm aff}$ is admissible, then $S$ is an $R$-basis of $I$, where
\begin{equation}\label{sbasis} S=\{ \boldf(x')^\a {x}^{\gamma}  c^{-1} e^f w d  \bar x^{\delta}\boldg(\bar x')^\b\in\mathcal N\mid \a_i+\b_j\neq 0 \text{ for some $i, j$}\}.\end{equation}
\end{lemma}
\begin{proof}
Let $M=\text{span}_R S$. 
By Lemma~\ref{mathcaln},  $\boldf(x_1)\mathscr B_{r, t}^{\rm aff}\subseteq M$.
For any positive integer $l$ with $1\le l<i$,
by  Lemma~\ref{lemm-----2}\,(2),  $$\boldf(x_i') \boldf(x_l')\in\mbox{$ \sum\limits_{j=1}^{i-1}$} \boldf(x_j')\mathscr B_{r,t}^{\rm aff}+\boldf(x_i') D,$$
such that $D\in \mathscr B_{r, t}^{\rm aff}$ and the degree of $D$ is strictly less then $k$. Thus, $\boldf(x_i')\mathscr B_{r, t}^{\rm aff}\subseteq M$ which follows from inductive assumption on $j$ with $ 1\le j\le i-1$
and inductive assumption on degrees. This proves $J_R\subseteq M$.   One can check $J_L \subseteq M$ similarly. By Proposition~\ref{2ij}\,(3),   $I=M$. \end{proof}

By abuse of notions, a regular monomial $\textbf{m}$ in Definition~\ref{def-mo} is also called  a {\it regular monomial} of  $\mathscr B_{k, r, t}$
if $0\le \alpha_i, \beta_j\le k-1$ for all $i, j$ with $1\le i\le r$ and $1\le j\le t$. Obviously, the number of all such regular monomials is  $k^{r+t}(r+t)!$.

\begin{theorem}\label{level-k-walled} The cyclotomic walled Brauer algebra  $\mathscr B_{k, r, t}$ is free over $R$ with rank  $k^{r+t}(r+t)!$ if and only if
 $\mathscr B_{k, r, t}$ is admissible.
\end{theorem}

\begin{proof} Let $M$ be the $R$-submodule of $\mathscr B_{k, r, t}$ spanned by all regular monomials of  $\mathscr B_{k, r, t}$.   By
induction on degrees, it is routine to check that  $M$ is left $\mathscr B_{k, r, t}$-module (cf.~\cite[Proposition~4.12]{RSu} for $\mathscr B^{\rm aff}_{r,t}$). Since $1\in M$, we have    $M=\mathscr B_{k, r, t}$.
If $\mathscr B_{k, r,t}$ is not admissible, by Lemma~\ref{lemm-level-k-walled},  $e_1$ is an $R$-torsion element. Since  $e_1\in M$, either $\mathscr B_{k, r, t}$ is not free or the rank of $\mathscr B_{k, r, t}$ is strictly less than  $k^{r+t}(r+t)!$. If  $\mathscr B_{k, r,t}$ is  admissible, by  Lemmas~\ref{mathcaln}--\ref{I},
the set of all regular monomials of  $\mathscr B_{k, r,t}$ is $R$-linear independent.  Thus, $\mathscr B_{k, r, t}$ is free over $R$ with rank  $k^{r+t}(r+t)!$.  \end{proof}

\section{ A weakly cellular basis of $\mathscr B_{2, r, t}$}
The aim of this section is to  construct a  weakly  cellular basis of
$\mathscr B_{2, r, t}$ in the sense of \cite{GG}. This basis will be used to set up a relationship between  $\mathfrak {gl}_{m|n}$-Kac-modules
and right cell modules of $\DBr$ in section~6.

Recall that a \textit{composition}  of $r$ is a sequence of non--negative integers $\tau=(\tau_1,\tau_2,\dots)$ such that
$|\tau|:=\sum_i\tau_i=r$. If $\tau_i\ge \tau_{i+1}$ for all possible $i$'s, then $\tau$ is called a \textit{partition}.
 Similarly, a \textit{$k$-partition}
of $r$, or simply a \textit{multipartition} of $r$, is an ordered $k$-tuple
$\lambda=(\lambda^{(1)},\lambda^{(2)},\cdots,  \lambda^{(k)})$ of partitions
with $|\lambda|:=\sum_{i=1}^k |\lambda^{(i)}|=r$.
Let $\Lambda^+_k(r)$ be the set of all $k$-partitions  of $r$. Let $\unlhd$ be the dominant order defined on $\Lambda_k^+(n)$ in the sense that $\lambda\unlhd\mu$ if and only if
\begin{equation}\label{par}\mbox{$
\sum\limits_{h=1}^{\ell-1} |\lambda^{(h)}|+\sum\limits_{j=1}^i \lambda_j^{(\ell)}\le  \sum\limits_{k=1}^{\ell-1} |\mu^{(h)}|+\sum\limits_{j=1}^i \mu_j^{(\ell)}$
 for $ \ell\le k$ and all possible $i$,}\end{equation} where $|\lambda^{(0)}|=0$.
 Then $\Lambda_k^+(r)$ is a poset with $\unlhd$ as a partial order on it.
In this paper, we always assume  $k\in \{1, 2\}$.

For each $\lambda\in \Lambda_1^+(r)$,  the {\it Young diagram} $[\lambda]$ is a
collection of boxes arranged in left-justified rows with $\lambda_i$
boxes in the $i$-th row of $[\lambda]$.  A {\it $\lambda$-tableau} $\s$ is
obtained by inserting elements $i,\, 1\le i\le r$ into $[\lambda]$ without
repetition. A $\lambda$-tableau $\s$   is said to be {\it standard} if the
entries in  $\s$   increase both from left to right in each row
and from top to bottom in each column. Let $\Std(\lambda)$ be the
set of all standard $\lambda$-tableaux.
Let $\t^\lambda\in \Std(\lambda)$  be obtained from  $[\lambda]$ by adding $1, 2, \cdots, r$ from left to right
along the rows of $[\lambda]$.
Let $\t_{\lambda}\in \Std(\lambda)$  be obtained from
$[\lambda]$ by adding $1, 2, \cdots, r$ from top to bottom along the columns of $[\lambda]$.
For example, if $\lambda=(3,2)$, then
\begin{equation}\label{tla0}
\t^{\lambda}= \ \ \young(123,45), \quad \text{ and \ }
 \t_{\lambda}=\ \  \young(135,24) .\end{equation}
If  $\lambda\in \Lambda_2^+(r)$,
then the corresponding Young diagram $[\lambda]$ is $([\lambda^{(1)}], [\lambda^{(2)}] )$.
In this case,  a {\it $\lambda$-tableau} $\s=(\s_1, \s_2)$ is
obtained by inserting elements $i,\, 1\le i\le r$ into $[\lambda]$ without
repetition. A $\lambda$-tableau $\s$   is said to be {\it standard} if the
entries in  $\s_i$, $1\le i\le 2$    increase both from left to right in each row
and from top to bottom in each column. Let $\Std(\lambda)$ be the
set of all standard $\lambda$-tableaux.
Let $\t^\lambda\in \Std(\lambda)$  be obtained from  $[\lambda]$ by adding $1, 2, \cdots, r$ from left to right
along the rows of $[\lambda^{(1)}]$ and then $[\lambda^{(2)}]$.
Let $\t_{\lambda}\in \Std(\lambda)$  be obtained from
$[\lambda]$ by adding $1, 2, \cdots, r$ from top to bottom along the columns of $[\lambda^{(2)}]$ and then
$[\lambda^{(1)}]$.
For
example, if $\lambda=((3,2), (3,1))\in \Lambda_2^+(9)$, then
\begin{equation}\label{tla}
\t^{\lambda}=\left( \ \ \young(123,45),\  \young(678,9)\ \ \right) \quad \text{ and \ }
 \t_{\lambda}=\left(\ \ \young(579,68), \ \young(134,2)\ \ \right).\end{equation}
Recall that  $\mathfrak S_r$ acts on the right of  $1, 2, \cdots, r$.  Then $\mathfrak S_r$  acts on the right of a $\lambda$-tableau
$\s$ by permuting its entries. For example, if $\lambda=((3,2), (3,1))\in \Lambda_2^+(9)$, and $w=s_1s_2$,  then
\begin{equation}\label{tlaw}
\t^{\lambda}w=\left( \ \ \young(312,45),\  \young(678,9)\ \ \right).\end{equation}
Write $d(\s)=w$ for $w\in\mathfrak S_r $ if $\t^\lambda w=\s$.
 Then  $d(\s)$ is
uniquely determined by $\s$. Let  $w_\lambda=d(\t_\lambda)$.
 The row stabilizer $\mathfrak S_\lambda$ of $\t^\lambda$ for $\lambda\in \Lambda_{k}^+(r)$ is known as the Young subgroup of $\mathfrak S_r$ with respect to $\lambda$. It is the same as the
  Young  subgroup  $\mathfrak S_{\l_{\rm comp}} $ with respect to the composition $\lambda_{\rm comp}$, which is
obtained from $\lambda$ by concatenation.
For example, $ \lambda_{\rm comp}=(3,2,3,1)$ if $\lambda=((3,2), (3,1))$.

The level two degenerate Hecke $\mathscr H_{2, r}$  with defining parameters $u_1$ and $u_2$ is $\mathscr H_{r}^{\rm aff}/I$, where $I $ is
the two-sided ideal of $\mathscr H_{r}^{\rm aff}$ generated by $(y_1-u_1)(y_1-u_2)$, $u_1, u_2\in R$. By definition, $\mathscr H_{2, r}$ is an $R$-algebra generated by $s_i$, $1\le i\le r-1$ and $y_j$, $1\le j\le r$ such that
\begin{enumerate}  \item  $s_is_j =s_js_i$,  $1<|i-j|$,
\item $ y_iy_\ell =y_\ell y_i$, $1\le i, \ell\le r$,
\item  $s_iy_i-y_{i+1}s_i=-1$,  $ y_is_i-s_iy_{i+1}=-1$,   $1\le i\le r-1$,
\item  $s_js_{j+1}s_j =s_{j+1}s_js_{j+1}$, $1\le j\le r-2$,
\item  $s_i^2 =1$, $1\le i\le r-1$,
\item $(y_1-u_1)(y_1-u_2)=0$. \end{enumerate}
 Following \cite{AMR}, we define   $\pi_\lambda=\pi_{a}(u_2)$ and $\tilde \pi_\lambda=\pi_{a}(u_1)$ for $\lambda\in \Lambda_2^+(r)$ with $|\lambda^{(1)}|=a$,
where for any $u\in R$,  $\pi_0(u)=1$ and $\pi_a(u)=\prod_{i=1}^a (y_i-u)$ if $a>0$.
Let
\begin{equation}\label{wa} w_a=\Big(\begin{array}{ccccccccc}1\!&\!2\!&\!\cdots\!&\!a\!&\!a\!+\!1\!&\!a\!+\!2\!&\!\cdots\!& r\\
r\!-\!a\!+\!1\!&\!r\!-\!a\!+\!3\!&\!\cdots\!&\! r \!& 1 \!&\!\!  2\!\!&\!\cdots\!&\!r\!-\!a\!\end{array}\Big).\end{equation}
 It is well-known that
 \begin{equation}\label{conjwa} w_a s_j = s_{(j)w_a^{-1}} w_a, \quad \text{if $j\neq r-a$}.\end{equation}
 Let $\mathfrak S_{a, r-a}$ be the Young subgroup with respect to the composition $(a, r-a)$. Then
  \equa{Young-a-r}{R\mathfrak S_{a, r-a} w_a =w_a R\mathfrak S_{r-a, a}.}
   For each composition $\lambda$ of $r$, we denote  \equa{x-lambda-y-}{\mbox{$x_\lambda =\sum\limits_{w\in\mathfrak S_\lambda} w ,\ \ y_\lambda= \sum\limits_{w\in\mathfrak S_\lambda} (-1)^{\ell(w)} w,$}} where
$\ell (\cdot )$ is the length function on $\mathfrak S_r$.
Assume $\lambda\in \Lambda_2^+(r)$ with $|\l^{(1)}|=a$. If we denote   $\mu^{(i)}=(\lambda^{(i)})'$, the conjugate of $\lambda^{(i)}$ for $i=1, 2$, then
   \begin{equation} \label{xy1}  w_{a} x_{\mu^{(2)}} y_{\mu^{(1)}}=y_{\mu^{(1)}} x_{\mu^{(2)}}  w_a.\end{equation}

\begin{remark}\label{com-conj}  When we write $x_{\mu^{(2)}} y_{\mu^{(1)}}$, then
$x_{\mu^{(2)}}$ (resp.,  $y_{\mu^{(1)}}$) is defined via symmetric group on $r-a$ letters $\{1, 2, \cdots, r-a\}$ (resp., on $a$ letters $\{r-a+1, \cdots, r\}$).
Similarly, when we write $y_{\mu^{(1)}} x_{\mu^{(2)}}$, then $y_{\mu^{(1)}}$ (resp., $x_{\mu^{(2)}}$) is defined via symmetric group on $a$ letters $\{1, 2, \cdots, a\}$
(resp., on $r-a$ letters $\{a+1, a+2, \cdots, r\}$). \end{remark}

\begin{definition}\label{h-cellu} For any $\s, \t\in \Std(\lambda)$ with  $\lambda\in \Lambda_2^+(r)$, define \begin{enumerate}
 \item $\mathfrak x_{\s \t}=d(\s)^{-1}  \mathfrak x_\l   d(\t)$,  where $\mathfrak x_\l =\pi_\lambda x_{\lambda^{(1)}} y_{\lambda^{(2)}}$, \item
 $\mathfrak y_{\s\t}=  d(\s)^{-1} \mathfrak y_\l d(\t)$, where $\mathfrak y_\l=  \tilde\pi_\lambda  x_{\lambda^{(1)}} y_{\lambda^{(2)} }$,
\item $\bar{\mathfrak x}_{\s\t}=d(\s)^{-1} \bar{ \mathfrak x}_\l   d(\t)$,  where $\bar{\mathfrak x}_\l =\pi_\lambda y_{\lambda^{(1)}} x_{\lambda^{(2)}}$,
\item  $\bar{\mathfrak y}_{\s\t}=d(\s)^{-1} \bar{\mathfrak y}_\l d(\t)$, where $\bar{\mathfrak y}_\l=  \tilde\pi_\lambda  y_{\lambda^{(1)}} x_{\lambda^{(2)} }$.
\end{enumerate}\end{definition}

It is proven in \cite{AMR} that $\mathscr H_{2, r}$ is a cellular algebra over $R$ in the sense of \cite{GL}. In this paper, we need the following cellular basis of $\mathscr H_{2, r}$ so as
to construct a new weakly cellular basis of $\mathscr B_{2, r, t}$.

 \begin{lemma}\label{cell-h2} The set $S_i$, $i\in \{1,2,3, 4\}$, are cellular bases of $\mathscr H_{2,r}$ in the sense of \cite{GL},  where
\begin{enumerate}\item  $S_1=\{\mathfrak x_{\s\t}\mid \lambda\in \Lambda_{2}^+(r), \s, \t\in \Std(\lambda)\}$,
\item  $S_2=\{\mathfrak y_{\s\t}\mid \lambda\in \Lambda_{2}^+(r), \s, \t\in \Std(\lambda)\}$,
\item $S_3=\{\bar {\mathfrak x}_{\s\t}\mid \lambda\in \Lambda_{2}^+(r), \s, \t\in \Std(\lambda)\}$,
\item  $S_4=\{\bar{\mathfrak y}_{\s\t}\mid \lambda\in \Lambda_{2}^+(r), \s, \t\in \Std(\lambda)\}$.
\end{enumerate}
\end{lemma}
\begin{proof} 

Let $S=\{x_{\s\t}\mid \s, \t\in \Std(\lambda), \lambda\in \Lambda_{2}^+(r)\}$  and $x_{\s \t}=d(\s)^{-1} \pi_\lambda x_{\lambda^{(1)}} x_{\lambda^{(2)}} d(\t)$.
It is proven in \cite{AMR} that $S$ is a cellular basis of  $\mathscr H_{2, r}$.
If we use  $y_{\lambda^{(2)}}$  instead of  $x_{\lambda^{(2)}}$  in $x_{\s\t}$, we will get $\mathfrak x_{\s\t}$.
However, for  any $\s=(\s_1, \s_2)\in \Std(\lambda)$, $d(\s)$  can be written uniquely as $d(\s_1)d(\s_2) d$ such that $d$ is a distinguished right coset representative of $\mathfrak S_a\times \mathfrak S_{r-a}$ in $\mathfrak S_r$ and  $\s_i\in \Std(\lambda^{(i)})$, where $a=|\l^{(1)}|$. So, the transition matrix between $S_1$   and   $S$
 is determined by the transition matrix between the cellular basis  $\{d(\s_2)^{-1} x_{\lambda^{(2)}} d(\t_2)\mid \lambda^{(2)}\in \Lambda^+(r-a), \s_2, \t_2\in \Std(\lambda^{(2)})\}  $ and  $\{d(\s_2)^{-1} y_{\lambda^{(2)}} d(\t_2)\mid \lambda^{(2)}\in \Lambda^+(r-a), \s_2, \t_2\in \Std(\lambda^{(2)})\}  $ of
   $R\mathfrak S_{r-a}$. Thus, $S_1$ is a basis of $\mathscr H_{2, r}$.
One  can check that  $S_1$ is a cellular basis of $\mathscr H_{2, r}$ in the sense of \cite{GL}  by mimicking Dipper-James-Murphy's arguments in the proof of Murphy basis for Hecke algebras of type $B$ in \cite{DJM}. We leave the details to the readers. Finally, (2)--(4) can be verified similarly.
\end{proof}

By Graham-Lehrer's results on the representation theory of cellular algebras in \cite{GL}, one can define right  cell modules of $\mathscr H_{2, r}$ via the cellular bases $ S_i$, $i\in \{1, 2, 3, 4\}$  in Lemma~\ref{cell-h2}.
The corresponding right cell modules of $\mathscr H_{2, r}$ with respect to  $S_2$ and $S_4$  are denoted by  $\tilde\Delta(\lambda)$, and $\bar \Delta(\l)$.

 For the simplification of discussion, we assume $\mathscr H_{2, r}$ is defined over $\mathbb C$ in Lemma~\ref{ineq}.

\begin{lemma}\label{ineq}  Suppose $a, b\in \mathbb N$. Then \begin{enumerate}\item
$\pi_a(u_2) \mathscr H_{2, r}  \pi_b(u_1) =0$ whenever $a+b> r$ and $a, b\in \mathbb Z^{>0}$.
\item $\pi_a(u_2) \mathscr H_{2, r} \pi_{r-a}(u_1)=  \pi_{a}(u_2) w_a \pi_{r-a}(u_1) \mathbb C\mathfrak S_{r-a, a}$, where $\mathfrak S_{r-a, a}$ is as in \eqref{Young-a-r}.
\item $\mathfrak x_\lambda \mathscr  H_{2, r} \mathfrak y_{\mu'}=0$ if $\lambda, \mu\in \Lambda_2^+(r)$ with  $\lambda\rhd \mu$,
\item  $\mathfrak x_\lambda \mathscr H_{2, r} \mathfrak y_{\l'}=\text{Span}_{\mathbb C}\{\mathfrak x_\lambda w_\lambda\mathfrak y_{\l'} \}$ if $\lambda\in \Lambda^+_2(r)$.
\item $\tilde\Delta(\lambda')\cong \mathfrak x_\lambda w_\lambda\mathfrak y_{\l'}\mathscr H_{2, r}$.\end{enumerate}  \end{lemma}
\begin{proof} (1)--(4)   can be proven by arguments similar to  those for Hecke algebras of type $B$ in \cite{DJ2}. We only give details for (3) and (5).

If $\lambda\rhd \mu$, then $|\lambda^{(1)}|\ge |\mu^{(1)}|$. If  $|\lambda^{(1)}|> |\mu^{(1)}|$, then $|\mu^{(1)}|\neq r$ and the result follows from
 (1). When $|\lambda^{(1)}|= |\mu^{(1)}|$,
by (2) together with corresponding result for
the group algebras of symmetric groups, we have $\lambda^{(i)}\unlhd \mu^{(i)}$ for $i=1, 2$ if  $\mathfrak x_\lambda \mathscr  H_{2, r} \mathfrak y_{\mu'}\neq 0$. This proves (3).

There is a surjective $\mathscr H_{2, r }$-homomorphism from $\phi: \mathfrak y_{\lambda'} \mathscr H_{2, r}\rightarrow \mathfrak x_\lambda w_\lambda\mathfrak y_{\l'}\mathscr H_{2, r}$.
Let $\mathscr H_{2, r}^{\rhd \lambda'}$ be the $\mathbb C$-submodule spanned by  $\{\mathfrak y_{\s\t}\mid \s, \t\in \Std(\mu), \mu\rhd \lambda'\}$. It follows from
standard results on cellular algebras that  $\mathscr H_{2, r}^{\rhd \lambda'}$  is a two-sided ideal of $\mathscr H_{2, r}$. So, $ \mathfrak y_{\lambda'}\mathscr H_{2, r}+\mathscr H_{2, r}^{\rhd \lambda'}/\mathscr H_{2, r}^{\rhd \lambda'}$ is isomorphic to a submodule of $\tilde \Delta(\lambda')$.
If  $\mathfrak y_{\s\t}\in\mathscr H_{2, r}^{\rhd \lambda'} $, we have $\mu\rhd\lambda'$ which is equivalent to $\lambda\rhd \mu'$. By (3),
$x_\lambda w_\lambda  \mathfrak y_{\s\t}=0$ and  $\mathscr H_{2, r}^{\rhd \lambda'}\subset \ker \phi$. So, there is an epimorphism from
$ \mathfrak y_{\lambda'}\mathscr H_{2, r}+\mathscr H_{2, r}^{\rhd \lambda'}/\mathscr H_{2, r}^{\rhd \lambda'} $ to  $\mathfrak x_\lambda w_\lambda\mathfrak y_{\l'}\mathscr H_{2, r}$. Mimicking arguments on classical
Specht modules for  Hecke algebra of type $B$ in \cite{DJ2}, we know that  $\mathfrak x_\lambda w_\lambda\mathfrak y_{\l'}\mathscr H_{2, r}$ has a basis  $\{\mathfrak x_\lambda w_\lambda\mathfrak y_{\l'}d(\t)\mid
\t\in \Std(\lambda')\}$. So, $$\dim_{\mathbb C}  \tilde \Delta(\lambda')   =\dim_{\mathbb C}\mathfrak x_\lambda w_\lambda\mathfrak y_{\l'}\mathscr H_{2, r}=\# \Std(\lambda'), $$  forcing
$ \mathfrak y_{\lambda'}\mathscr H_{2,r}+\mathscr H_{2, r}^{\unrhd \lambda'}/\mathscr H_{2, r}^{\unrhd \lambda'}\cong \mathfrak x_\lambda w_\lambda\mathfrak y_{\l'}\mathscr H_{2, r}\cong \tilde\Delta(\lambda')$.
\end{proof}

Now, we use  cellular bases  $S_i$ of $\mathscr H_{2, r}$ in Lemma~\ref{cell-h2} to construct a weakly cellular basis of $\mathscr B_{2, r, t}$ over an arbitrary field in the sense of \cite{GG}.
We remark that when we use results on  level two degenerate Hecke algebra for  $\mathscr B_{2, r, t}$, we should keep in mind that  $x_1, \bar x_1\in \mathscr B_{2, r t}$ should be regarded as  $-y_1\in \mathscr H_{2, r}$ and $\mathscr H_{2, t}$, respectively. Therefore, we have to use $-u_i$ and $-\bar u_i$ instead of $u_i$ and $\bar u_i$.

Fix $r, t, f\in \mathbb Z^{>0}$ with $f\le \min \{r, t\}$. In contrast to \eqref{rcs11}, we define
\begin{eqnarray} \label{rcs}\mathcal D_{r, t}^f=\{\! s_{r-f+1,i_{r-f+1}}{\sc\!} \bar s_{t-f+1, j_{t-f+1}}{\sc\!}\! \cdots\!
s_{r,i_r}{\sc\!}\bar s_{t,{j_t}}\mid
 r  {\!}\ge{\!} i_r{\!}>{\!}{\sc\!} \cdots{\sc\!}{\!} >{\!}i_{r-f+1}, j_k\ge k+f-t\}.  \end{eqnarray}
For each $c\in \mathcal D^f_{r, t}$ as in  \eqref{rcs}, let $\kappa_c$ be the
$r$-tuple \begin{equation}\label{k-ccc}
\kappa_c\!=\!(k_1,\dots,k_r)\!\in\!\{0,1\}^r\mbox{ such that $k_i\!=\!0$  unless $i\!=\!i_r,i_{r-1},\dots,i_{r-f+1}$.}\end{equation}
Note that $\kappa_c$ may have more than one choice for a fixed $c$, and it may
be equal to $\kappa_d$ although $c\neq d$ for $c, d\in \mathcal D^f_{r,
t}$. Let $\mathbf N_f\! =\!\{\kappa_c \,|\, c\!\in \!\mathcal D^f_{r,
  t}\}$. If  $\kappa_c\in \mathbf N_f$,  define  $x^{\kappa_c}=\prod_{i=1}^r x_i^{k_i}$.
In \cite{RSu}, we consider poset $(\Lambda_{2, r,t}, \unrhd)$, where
\begin{equation}\label{poset} \Lambda_{2, r,t} = \left\{ (f,\lambda, \mu )\,|\, (\lambda, \mu) \in \Lambda_2^+(r\!-\!f)\times \Lambda_2^+(t\!-\!f),\,  0\!\le\! f\! \le\! \min \{r,t \} \right\},\end{equation}
such that $(f,\lambda, \mu)\unrhd (\ell, \alpha, \beta)$ for $ (f, \lambda, \mu), (\ell, \alpha, \beta) \in  \Lambda_{2, r, t} $ if either $f>\ell$  or
$f=\ell$ and $\lambda\unrhd_1 \alpha$, and $ \mu\unrhd_2 \beta$, and in case $f=\ell$, the orders $\unrhd_1$ and $\unrhd_2$  are dominant orders on $\Lambda_2^+(r\!-\!f)$ and $\Lambda_2^+(t\!-\!f)$ respectively.
For each $(f, \mu, \nu)\in \Lambda_{2, r, t}$, let
\begin{equation} \label{index} \delta(f,\mu, \nu)
  =\{(\t, c, \kappa_c)\,|\,\t=(\t^{(1)}, \t^{(2)})\in\Std(\mu)\times \Std(\nu),c\in
                        \mathcal D^f_{r, t}\text{ and }\kappa_c\in\mathbf N_f
                        \}. \end{equation}

\begin{definition}\label{cellbasis}
For any $(\s, d, \kappa_d),(\t, c, \kappa_c)\in\delta(f,\mu, \nu)$ with  $(f,\mu, \nu)\in\Lambda_{2, r, t}$,    define
\begin{equation} \label{CCCCC} C_{(\s, d, \kappa_d)(\t,c, \kappa_c)}
              =x^{\kappa_d} d^{-1} \mathfrak e^{f} \n_{\s\t} c x^{\kappa_c}, \end{equation}
             where, in contrast to notation $e^f$ in \eqref{e-f==}, we define $\mathfrak e^f =e_{r, t} e_{r-1, t-1} \cdots e_{r-f+1, t-f+1}$ if  $f\ge1$  and  $\mathfrak e^0=1$, and
 $\n_{\s\t}= \mathfrak y_{\s^{(1)} \t^{(1)}}\bar {\mathfrak y}_{\s^{(2)}\t^{(2)}}$ if  $\s=(\s^{(1)}, \s^{(2)})$ and $\t=(\t^{(1)}, \t^{(2)})$ are in
 $ \Std(\mu)\times \Std(\nu) $.
\end{definition}

Note that   $\n_{\s\t}$ in Definition~\ref{cellbasis} are defined via  cellular basis elements of $\mathscr H_{2, r-f}$ and $\mathscr H_{2, t-f}$
 in Lemma~\ref{cell-h2}\,(2)\,(4). Since  $x_i$ and $\bar x_j$ do not commute each other,  a cellular basis element of  $\mathscr H_{2, r-f}$ is always put on the left.
 Further,  we need  to use $x_i$, $-u_1, -u_2$  (resp. $\bar x_i$, $-\bar u_1, -\bar u_2$) instead of $-y_i$, $u_1, u_2$ in  Lemma~\ref{cell-h2}.

\begin{theorem}\label{cellular-1} If  $\mathscr B_{2, r, t}$ is admissible, then
the set
$$\mathscr C=\{C_{(\s,\kappa_c,c)(\t,\kappa_d,d)}\,|\,
            (\s,\kappa_c,c),(\t,\kappa_d,d)\in\delta(f,\lambda),
              \forall (f,\lambda)\in\Lambda_{2, r, t}\}$$
is a weakly  cellular basis $\mathscr B_{2, r, t}$ over  $R$ in the sense of \cite{GG}.
\end{theorem}

\begin{proof} Let $S$ be the cellular basis of $\mathscr H_{2, r-f}$ (resp. $\mathscr H_{2, t-f}$) for $0\le f\le \min\{r, t\}$ defined in the proof of Lemma~\ref{cell-h2}. If we use $S$ instead of the cellular basis $S_2$ of $\mathscr H_{2, r-f}$  and $S_4$ of  $\mathscr H_{2, t-f}$ in Lemma~\ref{cell-h2}, we will obtain the weakly cellular basis of $\mathscr B_{2, r, t}$ over $R$ in \cite[Theorem~6.12]{RSu} provided that $R=\mathbb C$ and $u_1=-p$, $u_2=m-q$, $\bar u_1=q$ and
$\bar u_2=p-n$ with $r+t\le \min\{m, n\}$. Since  $\mathscr B_{2, r, t}$ is admissible, by  Theorem~\ref{level-k-walled}, the rank of  $\mathscr B_{2, r, t}$ is $2^{r+t}(k+t)!$. As pointed in \cite[Remark~6.13]{RSu}, \cite[Theorem~6.12]{RSu} holds over $R$ with arbitrary parameters $u_1, u_2, \bar u_1, \bar u_2$ if the rank of $\mathscr B_{2, r, t}$ is $2^{r+t} (r+t)!$.  Thus, $\mathscr C$ is an $R$-basis of  $\mathscr B_{2, r, t}$. Further, the weakly cellularity of $\mathscr B_{2, r, t}$ depends only on  cellular bases of $\mathscr H_{2, r-f}$ and $\mathscr H_{2, t-f}$ and does not depend on the explicit  descriptions of cellular bases of $\mathscr H_{2, r-f}$ and $\mathscr H_{2, t-f}$.
(cf.  the proof of  \cite[Theorem~6.12]{RSu}). So, all arguments  for the proof of \cite[Theorem~6.12]{RSu} can be used smoothly to prove that  $\mathscr C$ is a weakly  cellular basis $\mathscr B_{2, r, t}$ over  $R$. \end{proof}

Suppose $\mathscr B_{2, r, t}$ is defined  over a field $F$.
By  Theorem~\ref{cellular-1}, one can define  right cell
modules  $C(f, \mu, \nu)$ with respect to $(f, \mu, \nu)\in \Lambda_{2, r, t}$ for $\DBr$.
 Let    $\phi_{f,
\mu, \nu}$ be  the corresponding invariant form on $C(f, \mu, \nu)$ and let $D^{f, \mu, \nu} =C(f, \mu, \nu)/\text{Rad\,} \phi_{f, \mu, \nu}$, where $\text{Rad\,} \phi_{f, \mu, \nu}$ is the radical of $\phi_{f, \mu, \nu}$. By Graham-Lehrer's results  in \cite{GL} (a weakly cellular algebra has similar representation theory of a cellular algebra in \cite{GL}),
$D^{f, \mu, \nu}$ is either $0$ or irreducible and all non-zero  $D^{f, \mu, \nu}$ consist of a complete set of pair-wise non-isomorphic irreducible $\mathscr B_{2,  r, t}$-modules.
Let $\tilde \Delta(\mu)$ (resp. $\bar \Delta(\nu)$) be the cell module  of $\mathscr H_{2, r-f}$ (resp. $\mathscr H_{2, t-f}$) defined via
 $S_2$ and $S_4$ in Lemma~\ref{cell-h2}.  Similarly, one has the notations $D^\mu$ and $\bar D^\nu$, respectively.

\begin{proposition}\label{simp-b2}  Suppose that $\mathscr B_{2, r, t}$ is  admissible over $ F$.  For any $(f, \mu, \nu)\in \Lambda_{2, r, t}$,   $D^{f, \mu, \nu}\neq 0$  if and only if
\begin{enumerate}\item  $D^\mu\neq 0$ and $\bar D^\nu\neq 0$, \item $f\neq r $ provided $r=t$ and $\omega_0=\omega_1=0$.\end{enumerate}
\end{proposition}
\begin{proof} The  result can be proven  by arguments similar to those for Lemmas~7.3--7.4 in \cite{RSu}.\end{proof}

\begin{remark}\label{cycsim} By arguments similar to those for Theorem~\ref{cellular-1}, one can lift   cellular bases of $\mathscr H_{k, r}$ and $\mathscr H_{k, t}$ in \cite{AMR} to
obtain a weakly cellular basis of  $\mathscr B_{k, r, t}$ over $R$, provide that $\mathscr B_{k, r, t}$ is admissible. Further, it is not difficult to prove a result which is similar to
Proposition~\ref{simp-b2} for $\mathscr B_{k, r, t}$ over an arbitrary field $F$ with  characteristic $\text{char\,} F$ either zero or positive.
 Let $\mathbf u=(u_1, \cdots, u_k)\in F^k$  such that  $u_i=d_i\cdot1_F$ and  $0\le d_i<\text{char\,} F$ for $1\le i\le k$.  Kleshchev~\cite{K} has shown that the simple
$\mathscr H_{k,n}(\mathbf u)$-modules are \textit{labeled} by a set of multipartitions
which gives the same Kashiwara crystal as the set of $\mathbf u$-Kleshchev
multipartitions of~$n$ in  \cite{AM:simples,Ariki:class}. Thus,  the
simple $\mathscr B_{k, r, t}$-modules are labeled by the set
$\{(f,\mu, \nu )\}$, where (1) $0\le f\le\min\{r, t\}$, (2)  $\mu$'s  are
Kleshchev multipartitions of $r-f$ with respect to $\mathbf u $, (3) $\nu$'s are  Kleshchev multipartitions of $t-f$ with respect to  $\bar {\mathbf u}:=(\bar u_1, \bar u_2, \cdots, \bar u_k)$,
 (4) $f\neq r$ if $r=t$ and $\omega_i=0$ for $0\le i\le k-1$.
It is pointed in \cite{AMR} that one can  modify the proof of \cite[Theorem~1.1]{DM:Morita}, or~\cite[Theorem~1.3]{AM:simples},
to show  that  when $\mathscr B_{k, r, t}$ is admissible,  the
simple $\mathscr B_{k, r, t}$-modules are always labeled by the $(f, \mu, \nu)\in \Lambda_{k, r, t}$
with $0\le f\le \min\{r, t\}$ and $\mu$ (resp. $\nu$) are Kleshchev multipartitons with respect to $\mathbf u$ (resp. $\bar{\mathbf u}$)
and $f\neq r$ if $r=t$ and $\omega_i=0$ for $1\le i\le r$. However, we are not
claiming that $D^{(f,\mu, \nu)}\ne0$ for the multipartitions $\mu, \nu$
which Kleshchev~\cite{K} uses to label the  simple
$\mathscr H_{k,r-f}(\mathbf u)$-modules (resp. $\mathscr H_{k,r-f}(\bar{ \mathbf u})$-modules).
\end{remark}

 We recall the definition of Kleshchev bipartitions over $\mathbb C$ as follows (see e.g., \cite{Vazi}), which will be used in sections~5-6.
 Fix $u_1, u_2\!\in\! \mathbb C$ with $u_1\!-\!u_2\!\in\!\mathbb N$. Then   $\lambda\!=(\lambda^{(1)}, \lambda^{(2)})\!\in\! \Lambda^+_2(r)$   is called a {\it Kleshchev bipartition}~\cite{Vazi} with respect to $u_1, u_2$ if
\begin{equation}\label{Kle}
\lambda^{(1)}_{u_1-u_2+i}\le \lambda^{(2)}_{i} \text{ for all possible $i$.}\end{equation}
If $u_1-u_2\not\in \mathbb Z$,   all bipartitions of $r$ are Kleshchev bipartitions.
A pair of bipartitions  $(\mu, \nu)$ is {\it Kleshchev} if both $\mu$ and $\nu$ are Kleshchev  bipartitions  in the sense of \eqref{Kle} with respect to the parameters $u_1,\,u_2$ and $\bar u_1,\,\bar u_2$.
The following result will be used in section~6.
\begin{proposition} \label{classcell2} Suppose $\DBr$ is admissible over $\mathbb C$. For each $(f, \mu, \nu)\in \Lambda_{2, r, t}$, let $$\tilde C(f, \mu, \nu):=\mathfrak {e}^f \mathfrak x_{\mu'} \bar{ \mathfrak x}_{\nu'} w_{\mu'} w_{\nu'}
 \mathfrak {y}_{\mu}\bar{\mathfrak y}_\nu \mathscr B_{2, r,t} \pmod{ \mathscr \B_{2, r, t}^{f+1}},$$ where $\mathscr{B}_{2, r,t}^{f+1}$  is  the two-sided ideal of
$\mathscr{B}_{2, r,t}$ generated by $\mathfrak e^{f+1}$. Then   $C(f, \mu, \nu )\cong \tilde C(f, \mu, \nu)$.
\end{proposition}

\begin{proof}  Let $M_f$ be
the left  ${\mathscr{B}}_{2, r-f,t-f}$-module  generated by
\begin{equation} \label{key-2} V_{r,t}^f=\{ \mathfrak e^f d x^{\kappa_d}\mid (d, \kappa_d) \in \mathcal {D}_{r,t}^f\times \mathbf N_f \}.\end{equation}
By \cite[Proposition~6.10]{RSu},  $M_f=\mathfrak e^f \DBr$.
By \cite[Lemma~6.9]{RSu}, one can use $\mathscr H_{2, r-f}\otimes \mathscr H_{2,t-f}$ instead of   ${\mathscr{B}}_{2, r-f,t-f}$ in $\mathfrak x_{\mu'} \bar{ \mathfrak x}_{\nu'} w_{\mu'} w_{\nu'}
 \mathfrak {y}_{\mu}\bar{\mathfrak y}_\nu M_f  \pmod{ \mathscr \B_{2, r, t}^{f+1}}$. Now, the required isomorphism follows from Lemma~\ref{ineq}\,(5).
\end{proof}

\section{ Super Schur-Weyl duality}

The aim of this section is to generalize super Schur-Weyl duality between general linear Lie superalgebra $\mathfrak {gl}_{m|n}$ and $\DBr$ to the case $r+t>\min\{m, n\}$.
 Throughout, let $I_0=\{1,...,m\}$, $I_1=\{m+1,...,m+n\}$ and $I=I_0\cup I_1$.

For any pairs $(i, j)\in I\times I$,
let $E_{ij}$ be the matrix unit with parity $[E_{ij}]=[i]+[j]$, where $[i]=a$ if  $i\in I_a$, $a=0,1$.
The\textit{ general linear Lie superalgebra} $\mathfrak {gl}_{m|n}$ over $\mathbb C$, denoted by $\mfg$,  is
$\mfg_{-1}\oplus\mfg_0\oplus\mfg_1$, where
\begin{eqnarray}
&\!\!\!\!\!\!\!\!\!\!\!\!&\mfg_{-1}={\rm span}_\C\{E_{i,j}\,|\,i\in I_1,\,j\in I_0\},\ \
\mfg_{1}={\rm span}_\C\{E_{i,j}\,|\,i\in I_0,\,j\in I_1\},\nonumber\\
&\!\!\!\!\!\!\!\!\!\!\!\!&
\mfg_{0}={\rm span}_\C\{E_{i,j}\,|\,i,j\in I_0\mbox{ \ or \ }i,j\in I_1\}.
\end{eqnarray}
The Cartan subalgebra  $\fh$ of $\mfg$ is the $\mathbb C$-space with basis $\{E_{ii}\,\,|\,i\in I\}$. Let $\fh^*$ be the  dual space of $\fh$  with
dual basis $\{\es_i\,|\,i\in I\}$. Then any  $\xi\in\fh^*$, called a {\it weight} of $\mfg$, can be written as
\begin{equation}\label{wet}\mbox{$
\xi=\sum\limits_{i\in I_0} \xi^L_i\es_i+\sum\limits_{i\in I_1} \xi^R_{i-m}\es_i$} \text{ with $\xi^L_i,\xi_j^R\in\C$.}\end{equation}
Denote $\xi$ by  $(\xi^L_1,...,\xi^L_m\,|\,\xi^R_{1},...,\xi^R_{n})$.
If both  $\xi^L_i-\xi^L_{i+1}\in \mathbb N$ and $ \xi^R_j-\xi^R_{j+1}\in\!\mathbb N$ for all possible $i,j$, then $\xi$
is called {\it integral dominant}. Let  $P^+$ be the set of integral dominant weights.
 For any  $\xi\in P^+$,  let \begin{equation}\label {lrho} \xi^\rho:=\xi+\rho=(\xi^{L,\rho}_1,...,\xi^{L,\rho}_m\,|\,\xi^{R,\rho}_1,...,\xi^{R,\rho}_n),\end{equation}
where  $ \rho=(0,-1,...,1\!-\!m\,|\,m\!-\!1,m\!-\!2,...,m\!-\!n)$.
Following ~\cite{Kac77}, let
 $$\ell=\# \{(i,j)\mid  \xi^{L,\rho}_i\!+\!\xi^{R,\rho}_j\!=\!0, 1\le i\le m, 1\le j \le n\}.$$
Then $\xi$  is called  an {\it $\ell$-fold atypical weight} if $\ell>0$. Otherwise,
$\xi$ is called  {\it a typical weight}.

\begin{example} \label{l-pq2} For any $p, q\in\C$, let $\l_{pq}=(p,...,p\,|\,-q,...,-q)$.
Then $\l_{pq}$ is a typical  weight if and only if
\begin{equation}\label{Do-condi}
\mbox{$p - q \notin \Z$ \ or \ $p - q \le -m$ \ or \ $p - q \ge n$.}\end{equation}\end{example}
The current $q$ should be regarded as  $q+m$ in \cite[IV]{BS4}. In the remaining part of this paper,  $\l_{pq}$ is  always a typical  weight in the sense of \eqref{Do-condi}.

Let  $V=\C^{m|n}$ be the
natural $\mfg$-module with natural basis $\{v_i\,|\,i\in I\}$
such that $v_i$ has parity $[v_i]=[i]$. Then the dual space   $V^*$, which  has  the dual basis $\{\bar v_i\,|\,i\in I\}$, is a left $\mfg$-module  such that

\begin{equation}\label{dua} E_{ab}\bar v_i=-(-1)^{[a]([a]+[b])}\d_{ia}\bar v_b \text{ for any $(a, b)\in I\times I$}.\end{equation}
In particular, the weight of $\bar v_i$ is $-\epsilon_i$.
For the simplicity of notation, we set $W=V^*$. 

\begin{definition} \label{mlpq}
Fix  $r,t\!\in\!\Z^{>0}$.
Let  $V^{rt}\!=\!V^{\otimes r}\OTIMES W^{\otimes t}$ and
$ M_{pq}^{rt}\!=\!V^{\otimes r}\OTIMES K_{\l_{pq}}\OTIMES W^{\otimes t}$,
where $K_{\l_{pq}}$ is the Kac-module~\cite{Kac77} with respect to the highest weight  $\lambda_{pq}$ in Example~\ref{l-pq2}. \end{definition}

Let  $\pi:M_{pq}^{rt}\to V^{rt}$  be the projection such that, for any  $v\in M_{pq}^{rt}$,   $\pi(v)$ is the vector obtained from $v$ by deleting the tensor factor
in $K_{\l_{pq}}$. Let $v_{pq}$ be the  highest weight vector of $K_{\l_{pq}}$ with highest weight $\lambda_{pq}$. Then $v_{{pq}}$ is unique  up to a scalar. It is well-known (e.g. see  \cite{BS4}) that
$K_{\l_{pq}}$ is $2^{mn}$-dimensional with a basis
\begin{equation}\label{basis-k-l}
B=\Big\{b^\si:=\mbox{$\prod\limits_{i=1}^{n}\,\prod\limits_{j=1}^{m}$}E_{m+i,j}^{\si_{ij}}v_{{pq}}\,\Big|\,\si=(\si_{ij})_{i,j=1}^{n,m}\in\{0,1\}^{n\times m}\Big\},
\end{equation}

where the products are taken in any fixed order. Define 
\begin{equation} \begin{aligned}\label{imn} I(m|n, r)=&\{\mathbf i\mid \mathbf i=(i_r, i_{r-1},
\cdots, i_{1}), \,   i_j\in I,\,  1\le
j\le r\}, \\
\bar I(m|n, t)=&\{\mathbf j\mid \mathbf j=(j_{1}, j_{ 2},
\cdots, j_{t}), \,   j_{ i}\in I,\,  1\le
i\le t\}.
\end{aligned}
\end{equation}
If  $(\mathbf i, b, \mathbf j)\in I(m|n, r)\times B\times \bar I(m|n, t)$,  we define
\begin{equation}\label{iaj}
v_{\mathbf i, b, \mathbf j} =v_{i_r}\otimes v_{i_{r-1}}\otimes \cdots\otimes  v_{i_1}\otimes b\otimes
\bar v_{j_{ 1}}\otimes \bar v_{j_{ 2}}\otimes \cdots\otimes  \bar v_{j_{ t}}\in M_{pq}^{rt}.
\end{equation}

\begin{lemma} \label{basis-M-} Let $B_M=\{v_{\mathbf i}\otimes b\otimes  \bar v_{\mathbf j}\mid (\mathbf i, b, \mathbf j) \in I(m|n,
r)\times B \times \bar I(m|n, t)\}$. Then $B_M$ is a basis of  $M^{rt}_{pq}$.\end{lemma}

Denote by
$U(\mfg)$  the universal enveloping algebra of $\mfg$. Then $M_{pq}^{rt}$ is a
left $U(\mfg)$-module.
Let $J=J_1\cup \{0\}\cup J_2$ with $J_1=\{r,...,2,1\}$ and $J_{2}=\{\bar 1, \bar 2, ..., \bar t\}$.
Then $(J, \prec)$ is a total ordered set with $$r\prec r-1\prec \cdots\prec 1\prec 0\prec \bar 1\prec\cdots\prec \bar t.$$
For any  $a,b\in J$ with $a\prec b$,  define
$\pi_{ab}:U(\mfg)^{\otimes2}\to U(\mfg)^{\otimes(r+t+1)}$ by \begin{equation}\label{pi-ab}
\pi_{ab}(x\OTIMES y)=1\OTIMES\cdots\OTIMES1\OTIMES \overset{a\text{-th}} {x}\OTIMES 1\OTIMES\cdots\OTIMES1\OTIMES\overset{b\text{-th}} y\OTIMES1\OTIMES\cdots\OTIMES1.\end{equation}
Let $\Omega$ be a  {\it Casimir element} in $\mfg^{\otimes2}$ given by  \begin{eqnarray}\label{def-Omega}
\Omega=\SUM{i,j\in I}{}(-1)^{[j]}E_{ij}\OTIMES E_{ji}.\end{eqnarray}
In \cite{RSu}, we define operators  $s_i$, $\bar s_j$, $x_1$, $\bar x_1$ and $e_1$ acting on the right of $M_{pq}^{rt}$
via the following formulae:
\begin{eqnarray}\label{operator--1}&\!\!\!\!&
s_i=\pi_{i+1,i}(\Omega)|_{M_{pq}^{rt}}\ (1\le i<r),\ \ \ \
\bar s_j=\pi_{\bar j,\overline{j+1}}(\Omega)|_{M_{pq}^{rt}}\ (1\le j<t),\nonumber\\&\!\!\!\!&
x_1=-\pi_{10}(\Omega)|_{M_{pq}^{rt}},\ \ \ \ \ \ \ \bar x_1=-\pi_{0\bar1}(\Omega)|_{M_{pq}^{rt}},\ \ \ \ \ \ \
  e_1=-\pi_{1\bar1}(\Omega)|_{M_{pq}^{rt}}.
\end{eqnarray}
Then there is an algebra homomorphism $\phi:\mathscr B_{2, r, t}\rightarrow \text{End}_{U(\mfg)} (M_{pq}^{rt})$ sending the generators $s_i, \bar s_j$, $x_1$, $\bar x_1$ and $e_1$ to the operators
 $s_i$, $\bar s_j$, $x_1$, $\bar x_1$ and $e_1$ as above~\cite{RSu}. In this case, we need to use $-p, m-q$, and $q, p-n$ instead of $u_1, u_2$, $\bar u_1$ and $\bar u_2$  respectively in Definition~\ref{cycwb}  for $k=2$. Further, $\omega_0=m-n$, $\omega_1=nq-mp$ and $\omega_a=(m-p-q)\omega_{a-1}-p(q-m)\omega_{a-2}$  for $a\ge2$ and $\bar\omega_a$'s are  determined by
\cite[Corollary~4.3]{RSu}. Thus, $\DBr$ is admissible in the sense of Definition~\ref{condi-k}. By Theorem~\ref{level-k-walled}, $\dim_\mathbb C \DBr=2^{r+t}(r+t)!$.
We will always consider  $\mathscr B_{2, r, t}$ as above  in the remaining part of this paper.

\begin{theorem}\label{level-21}\cite[Theorem~5.16]{RSu} Fix $r, t\in \mathbb Z^{>0}$ with  $r+t\le \min\{m, n\}$. Then ${\rm End}_{\mfg}(M_{pq}^{rt})\cong \DBr$.
\end{theorem}

\begin{theorem}\cite[IV, Theorem~3.13]{BS4}\label{bshec}  If  $0<r\le \min\{m, n\}$, then $\text{End}_{U(\mfg)} (M_{pq}^{r0} )\cong \mathscr H_{2, r}$, the level two Hecke algebra with defining parameters $u_1=-p$ and $u_2=m-q$. \end{theorem}

\begin{theorem}\label{level-2schur} {\rm(Super Schur-Weyl duality)} Keep the condition~\eqref{Do-condi}.
 The algebra homomorphism $\phi_1:\DBr\twoheadrightarrow{\rm End}_{\mfg}(M_{pq}^{rt})$ is surjective. It it injective if and only $r+t\le\min\{m,n\}$.
\end{theorem}
\begin{proof}  By Theorem~\ref{level-21}, it suffices to prove that $\phi_1$ is surjective and is not injective if $r+t>\min\{m, n\}$.
As in \cite[(7.16)]{BS}, the map ${\rm flip}_{r,t}$ defined by the  following commutative diagram is a $\mfg$-module isomorphism
\begin{eqnarray}\label{diag1}
\begin{array}{ccccc}{\rm End}_\C(V^{\otimes r}\OTIMES K_{\l_{pq}}\OTIMES (V^*)^{\otimes t})
\!&\!\stackrel{{\rm flip}_{r,t}}{-\!\!\!-\!\!\!-\!\!\!-\!\!\!-\!\!\!-\!\!\!-\!\!\!\longrightarrow}\!&\!{\rm End}_\C(V^{\otimes r}\OTIMES K_{\l_{pq}}\OTIMES V^{\otimes t})\\
\mbox{\large$|{\sc\!}|$}\!&\!\!&\!\mbox{\large$|{\sc\!}|$}\\
{\rm End}_\C(V^{\otimes r}\OTIMES K_{\l_{pq}})\OTIMES{\rm End}_\C((V^*)^{\otimes t})\!&\!\!\!\stackrel{f\otimes g^*\mapsto f\otimes g}{-\!\!\!-\!\!\!-\!\!\!-\!\!\!-\!\!\!-\!\!\!-\!\!\!\longrightarrow}\!\!\!&\!
{\rm End}_\C(V^{\otimes r}\OTIMES K_{\l_{pq}})\OTIMES{\rm End}_\C(V^{\otimes t}).
\end{array}
\end{eqnarray}
Note that   ${\mathscr H}_{2,r+t}$ (denoted as $H_{r+t}^{p,q}$ in \cite[IV]{BS4}) is a subspace of ${\rm End}_\C(V^{\otimes r}\OTIMES K_{\l_{pq}}\OTIMES V^{\otimes t})$, thus \eqref{diag1} induces the following commutative diagram
\begin{eqnarray}\label{diag2}
\begin{array}{ccccc}\DBr
\!&\!\stackrel{{\rm flip}_{r,t}}{-\!\!\!-\!\!\!-\!\!\!-\!\!\!-\!\!\!\longrightarrow}\!&\!{\mathscr H}_{2,r+t}\\
\!\!\!\!\!\!\phi_1\ \ \put(4,10){\vector(0,-1){20}}\!&\!\!&\! \put(4,10){\vector(0,-1){20}}\ \ \pi_1\\
{\rm End}_\C(V^{\otimes r}\OTIMES K_{\l_{pq}}\OTIMES (V^*)^{\otimes t})\!&\!\stackrel{{\rm flip}_{r,t}}{-\!\!\!-\!\!\!-\!\!\!-\!\!\!-\!\!\!\longrightarrow}\!&\!{\rm End}_\C(V^{\otimes r}\OTIMES K_{\l_{pq}}\OTIMES V^{\otimes t}).
\end{array}
\end{eqnarray}
 By Theorem~\ref{level-k-walled} for $k=2$,
$\dim_\mathbb C\DBr=2^{r+t} (r+t)!$.  This implies that the top map is a bijection, and the  bottom map is a $\mfg$-module isomorphism, which induces an isomorphism between two subspaces
${\rm End}_\mfg(V^{\otimes r}\OTIMES K_{\l_{pq}}\OTIMES (V^*)^{\otimes t})$ and ${\rm End}_\mfg(V^{\otimes r}\OTIMES K_{\l_{pq}}\OTIMES V^{\otimes t})$. Since $\pi_1$ is surjectively mapped to ${\rm End}_\mfg(V^{\otimes r}\OTIMES K_{\l_{pq}}\OTIMES V^{\otimes t})$ by \cite[IV, Theorem 3.21]{BS4}, we see that $\phi_1:\DBr\twoheadrightarrow{\rm End}_{\mfg}(M_{pq}^{rt})$ is surjective.
Finally, the second assertion follows from the corresponding result for $t=0$ in~\cite[IV, Theorem~3.21]{BS4}.
\end{proof}

\def\u{\mathfrak u}
\def\v{\mathfrak v}

\section{Highest weight   vectors in $V^{\otimes r} \otimes K_{\lambda_{pq}}$} \label{hhwv}
The aim of this section is to give a  classification of  highest weight   vectors of  $M_{pq}^{r0}:=V^{\otimes r} \otimes K_{\lambda_{pq}}$ when $r\le \min\{m, n\}$,
 where $V$ is the natural representation of $\mfg:=\mathfrak {gl}_{m|n}$ and $K_{\lambda_{pq}}$ is the Kac-module with highest weight $\lambda_{pq}$ in Example~\ref{l-pq2}.
 This will be done in a few steps. First, by noting that $\mathfrak{g}$-highest weight vectors of $M_{pq}^{r0} $  is in one to one correspondence with the $\mathfrak{g}_{ 0}$-highest weight vectors of $V^{\otimes r}$ (cf.~\cite[Lemmas~5.1-5.2]{SHK}), we are able to reduce the problem to
the Lie algebra case. Secondly, since $\mathfrak{g}_{0}=\mathfrak{gl}_m\oplus\mathfrak{gl}_n$,  and $V^{\otimes r}$ can be decomposed as a direct sum of tensor products
of natural representations of $\mathfrak{gl}_m$ and $\mathfrak{gl}_n$,
we are able to further simplify the problem to the $\mathfrak{gl}_m$ case.

To begin with, we briefly  recall the results on a classification of $\mathfrak{gl}_m$-highest weight vectors of $V^{\otimes r}$, where $V$ temporarily denotes the natural
representation of $\mathfrak{gl}_m$ over $\mathbb C$.
Let $\{v_i\mid 1\le i\le m\}$ be a basis of $V$. Obviously, $V^{\otimes r}$ has a basis
$\{v_{\mathbf i}\mid \mathbf i\in I(m|0, r)\}$,  where
$$v_{\mathbf i}=v_{i_r}\otimes v_{i_{r-1} }\otimes \cdots \otimes v_{i_1}.$$
We consider a Cashmir element $\Omega$ in $\mathfrak {gl}_m^{\otimes2} $ with
\begin{equation}\label{evenomega}\mbox{$
 \Omega=\sum\limits_{1\le i,j\le m} E_{ij}\OTIMES E_{ji}\in \mathfrak {gl}_m^{\otimes2}, $}\end{equation}
 which is a special case of \eqref{def-Omega}.
 Define $\mathbf s_i=\pi_{i,i+1}(\Omega)$, $1\le i\le r-1$.  Then   $(i, i+1)\in \mathfrak S_r$
acts   on $V^{\otimes r}$  via  $\mathbf s_i$. Thus,  $V^{\otimes r }$ is a $(\mathfrak {gl}_m, \mathbb C\mathfrak S_{r})$-bimodule such that
 \begin{equation} \label{placep} v_{\mathbf i} w=v_{i_{(r)w^{-1} }} \otimes v_{i_{(r-1)^{w^{-1}}}} \otimes \cdots \otimes v_{i_{(1)^{w^{-1}}}} \text{ for any  $w\in \mathfrak S_r$.}\end{equation}
For example, $v_{i_3} \otimes v_{i_2} \otimes v_{i_1} s_1s_2=v_{i_1}\otimes v_{i_3}\otimes v_{i_2}$.
If $r\le m$, it is well-known that  $$\text{End}_{U(\mathfrak {gl}_m)}(V^{\otimes r})\cong \mathbb C\mathfrak S_r.$$

\begin{definition}\label{vlambda} If $\lambda\in\Lambda^+(r, m)$, the set of partitions of $r$ with at most $m$ parts,   we define
$v_\lambda= v_{\mathbf i_\lambda}\in V^{\otimes r}$,  where $\mathbf i_\lambda=(1^{\lambda_1}, 2^{\lambda_2}, \cdots, m^{\lambda_m})$ and $k^{\lambda_k}$ denotes the sequence $k, k, \cdots, k$ with multiplicity $\lambda_k$.
\end{definition}
The following  result is  well-known,  and  Lemma~\ref{liehwv} follows from Lemma~\ref{ine123}.

\begin{lemma}\label{ine123}Suppose $\lambda$ and $\mu$ are two compositions of $r$ and $\mu'$ is the conjugate of $\mu$, and $x_\l,y_{\mu'}$ are defined in \eqref{x-lambda-y-}. Then  $x_\lambda \mathbb C \mathfrak S_r y_{\mu'}=0$ unless $\lambda\unlhd \mu$.
\end{lemma}

\begin{lemma}\label{liehwv} There is a bijection between the set of  dominant weights of $V^{\otimes r}$ and $\Lambda^+(r, m)$, the set of partitions of $r$ with at most $m$ parts.
Further, the $\mathbb C$-space of $\mathfrak{gl_m}$-highest weight vectors with highest weight $\lambda$ has a basis
$\{ v_{\lambda} w_\lambda y_{\lambda'} d(\t)\mid \t\in \Std(\lambda')\}$.
 \end{lemma}

Now, we  turn to construct $\mfg$-highest weight vectors  of $M_{pq}^{r0}$. Since  $r\le \min\{m, n\}$,
 there is a bijection between the set of  dominant weights of $M_{pq}^{r0}$
and  $\Lambda_{2}^+(r)$. Further, if $\lambda=(\lambda^{(1)}, \lambda^{(2)})\in \Lambda_2^+(r)$, the corresponding dominant weight of  $M_{pq}^{r0}$ is
 \begin{equation}\label{barl} \bar \lambda:=\lambda_{pq}+\tilde \lambda,\end{equation}
 where
\begin{equation}\label{tilla} \tilde \lambda=(\lambda^{(1)}_1, \cdots, \lambda^{(1)}_m\mid \lambda^{(2)}_1, \cdots, \lambda^{(2)}_n).\end{equation}
For instance, if $\lambda=((3, 1),(2,1))$, then $\tilde \lambda=(3, 1, 0, \cdots, 0\mid 2, 1, 0, \cdots, 0)$.
Recall that $\Omega$ is  a   Casimir element in $\mfg^{\otimes2}$ given in \eqref{def-Omega}.
Define operators  $s_i$, $x_1$  acting on the right of $M_{pq}^{r0} $
via the following formulae: $s_i=\pi_{i+1,i}(\Omega)$, $1\le i\le r-1$ and $x_1=-\pi_{10}(\Omega)$. In this case, $u_1=-p$ and $u_2=m-q$.
We remind that Brundan-Stroppel \cite{BS4} defined $x_1$ via $\pi_{10}(\Omega)$. So, the current $x_1$ is  $-x_1$ in \cite{BS4}.
 Recall that $v_{\mathbf i}\otimes v_{pq}=v_{i_r}\otimes \cdots \otimes v_{i_2}\otimes v_{i_1}\otimes v_{pq}$ for any $\mathbf i \in I(m|n, r)$ (cf. \eqref{imn}), and  $x_k'=x_k+L_k$ with $L_k=\sum_{i=1}^{k-1} (i, k)$ (see
 Lemma~\ref{x-p-iii}).

\begin{lemma}\cite[Lemma~3.1]{BS4} \label{xsfor} Suppose $\mathbf i\in I(m|n, r)$, and $1\le k\le r$.
\begin{enumerate} \item
$v_{\mathbf i}\otimes v_{pq} x_k'=-p v_{\mathbf i}\otimes v_{pq}$  if $1\le i_k\le m$.
\item $v_{\mathbf i}\otimes v_{pq} x_k'=-qv_{\mathbf i}\otimes v_{pq}+\sum_{j=1}^m (-1)^{\sum_{l=1}^{k-1} [i_l]} v_{\mathbf j} \otimes (E_{i_k, j} v_{pq})$
if $m+1\le i_k\le m+n$, where   $\mathbf j\in I(m|n, r)$ which is obtained from $\mathbf i$ by using $j$ instead of $i_k$ in $\mathbf i$. In particular, the weight of $v_{\mathbf j}$ is strictly bigger than that of $v_{\mathbf i}$.
 \end{enumerate}
\end{lemma}

\begin{definition}\label{muh}  For $\lambda\!=\!(\lambda^{(1)}, \lambda^{(2)}) \in \Lambda_2^+(r)$,  define $v_{\tilde\lambda}\!=\!v_{\mathbf i}$ with
$\mathbf i\!=\!(\mathbf i_{\l^{(1)}}, \mathbf i_{\l^{(2)}})\in I(m|n, r)$.
\end{definition}

For instance, $v_{\tilde \lambda}=v_{\mathbf i}$ if $\lambda=((3, 1),(2,1))$, where $\mathbf i=(1^3, 2, (m+1)^2, m+2)$.
\begin{definition}\label{vt1} For any $\t\in \Std(\lambda')$, we   define  $v_\t= v_{\tilde\lambda }\otimes v_{pq} w_{\lambda}\mathfrak y_{\l'} d(\t)$, where
$\mathfrak y_{\l'}$ is given in Definition~\ref{h-cellu}\,(2).\end{definition}

\begin{theorem} \label{hecg} Suppose $r\le \min\{m, n\}$.  There is a bijection between the set of  dominant weights of $M_{pq}^{r0}$ and  $\Lambda_{2}^+(r)$.
Further, the $\mathbb C$-space $V_{\bar \l}$ of $\mfg$-highest weight vectors of $M_{pq}^{r0}$ with highest weight $\bar\lambda$ has a basis
$\{ v_\t \mid \t\in \Std(\lambda')\}$.
\end{theorem}
\begin{proof} The required bijection between $\Lambda_{2}^+(r)$ and the set of  dominant weights of $M_{pq}^{r0}$
is the map sending $\lambda$ to $\bar\lambda$ defined in \eqref{barl}.
We claim that each $ v_\t$  is killed by $E_{m, m+1}$ and $E_{i, j}$ with $i<j$
and either $i, j\in I_0$ or $i, j\in I_1$. Since $M_{pq}^{r0} $ is $(\mfg,\mathscr H_{2, r})$-bimodule, we need only consider the case $d(\t)=1$. In this case,  $\t=\t^{\lambda'}$.

Denote $|\lambda^{(1)}|=a$. Recall that   $w_{\lambda^{(1)}}\in \mathfrak S_a$ and  $w_{\lambda^{(2)}}\in \mathfrak S_{r-a}$ such that  $\t^{\lambda^{(i)}} w_{\lambda^{(i)}}=\t_{ \lambda^{(i)}}$ for $i=1, 2$. Then
  \begin{equation}\label{wla1} w_\lambda =w_{\lambda^{(1)}} w_{\lambda^{(2)}} w_a=w_a  w_{\lambda^{(2)}} w_{\lambda^{(1)}}.  \end{equation}
By \eqref{conjwa} and \eqref{wla1},   $$v_{\t} = v_{\tilde\lambda }\otimes v_{pq} w_{\lambda^{(1)}} w_{\lambda^{(2)}} y_{\mu^{(1)}} x_{\mu^{(2)}} w_a \pi_{r-a}(-p),$$
where $\mu^{(i)}$ is the conjugate of $\l^{(i)}$ for $i=1, 2$.
By  Lemmas~\ref{ine123}--\ref{liehwv}, $v_{\t} $  is killed by  $E_{i, j}$ with $i<j$ and
either $i,j\in I_0$ or $i, j\in I_1$.
Since  $E_{m, m+1}$ acts on $M_{pq}^{r0} $ via $\sum_{i=1}^{r+1} 1^{\otimes i-1}\otimes E_{m, m+1}\otimes 1^{\otimes r+1-i}$, we have  $E_{m, m+1} v_{\tilde \lambda}\otimes v_{pq}=0$ if $v_{m+1}$ does not occur in $v_{\tilde \lambda}$. Otherwise, $\lambda^{(2)}\neq \emptyset$ and $r-a\neq 0$. In this case, up to a sign, $E_{m, m+1} v_{\tilde \lambda} \otimes v_{pq}$ is equal to
 $$v_{\mathbf j} \otimes v_{pq}(1-s_{a+1}+s_{a+1, a+3}+\cdots+ (-1)^{b-a} s_{a+1, b+1}),$$
 where $b=a+\lambda^{(2)}_1-1$ and $v_{\mathbf j}$ is obtained from $v_{\tilde \lambda}$ by replacing $v_{m+1}$ by $v_{m}$ at the $(a+1)$-th position. Thus,  $j_{a+1}=m$.
 Let  $$h=(1-s_{a+1}+s_{a+1, a+3} +\cdots+ (-1)^{b-a} s_{a+1, b+1}) w_{\lambda^{(1)}} w_{\lambda^{(2)}} y_{\mu^{(1)}} x_{\mu^{(2)}}.$$
Then  $h\in \mathbb C\mathfrak S_a\otimes \mathbb C\mathfrak S_{r-a}.$ By \eqref{conjwa}, $h w_a=w_a h_1$ for some  $h_1\in
 \mathbb C\mathfrak S_{r-a}\otimes \mathbb C \mathfrak S_{a}$. Since $h_1 \pi_{r-a}(-p)=\pi_{r-a}(-p) h_1$, it is enough to prove
 $v_{\mathbf j} \otimes v_{pq} w_a \pi_{r-a}(-p)=0$. Up to a sign,   $v_{\mathbf j} \otimes v_{pq} w_a =v_{\mathbf k}\otimes v_{pq}$ for some $\mathbf k$ such that  $v_{k_1}=v_m\in V_0$.
Since $r-a\neq 0$,  $x_1+p$ is a factor of $\pi_{r-a}(-p)$. By Lemma~\ref{xsfor}\,(1), $v_{\mathbf j} \otimes v_{pq} w_a \pi_{r-a}(-p)=0$. Thus,  $v_{\t} $ is a highest weight vector of $M_{pq}^{r0}$ if $v_\t\neq 0$.

 Note that any vector of $M_{pq}^{r0}$ can be written as $v=\sum_{b\in B}v_{b}\otimes b$, where $B$ is a basis of $K_{\l_{pq}}$ defined in \eqref{basis-k-l} and  $v_b\in V^{\otimes r}$. Following \cite{BS4},  $v_{b}$ is called   the {\it $b$-component} of $v$.
By Lemma~\ref{xsfor}\,(2) (or the arguments in the proof of \cite[Corollary~3.3]{BS}), the $v_{pq}$-component of  $v_{\tilde\lambda }\otimes v_{pq} w_{a}\pi_{r-a}(-p)$ is
 $v_{\tilde\lambda } w_{a} \prod_{i=1}^{r-a} (p-q-L_i)$.  By Lemma~\ref{x-p-iii}\,(3),  $\prod_{i=1}^{r-a} (p-q-L_i)$ is  a  central element in $\mathbb C\mathfrak S_{r-a}$, which   acts on $v_{\tilde\lambda }\otimes v_{pq} w_{\lambda^{(2)}} x_{\mu^{(2)}}$ as scalar
     $\prod_{i=1}^{r-a} (p-q-\res_{\t^{\mu^{(2)}}}  (i))$, where  $\mu=\l'$ and   $ \res_{\t^{\mu^{(2)}}}(i) $ is $j-l$ if $i$ is in the $l$-th row and  $j$-th column of  $\t^{\mu^{(2)}}$.  Since  $\lambda_{pq}$ is typical (cf. \eqref{Do-condi}), and  $r\le \min\{m, n\}$,
    $\prod_{i=1}^{r-a} (p-q-\res_{\t^{\mu^{(2)}}}  (i))\neq 0$.  So, \begin{equation} \label{vpqt} \text{the $v_{pq}$-component of
    $ v_{\t} $  \ = \  $v_{\tilde\lambda } w_a w_{\lambda^{(2)}} x_{\mu^{(2)}} w_{\lambda^{(1)}} y_{\mu^{(1)}} d(\t) $ (up to a non-zero scalar).}\end{equation}
     By Lemma~\ref{liehwv}, it is a $\mfg_0$-highest vector of $V^{\otimes r}$
    with highest weight $\tilde \l$ (cf.~\eqref{tilla}), forcing $v_\t\neq 0$.

    Now, we prove that $\{ v_{\t} \mid \t\in \Std(\lambda')\}$ is $\mathbb C$-linear independent.
    First, consider $V=V_0\oplus V_1$ as a module for $\mfg_0=\mathfrak{gl}_m\oplus \mathfrak{gl}_n$. Then $V^{\otimes r }$ can be decomposed as a direct sum of $V_{i_1}\otimes V_{i_2}\otimes\cdots \otimes V_{i_r}$, where
    $i_j\in \{0, 1\}$. As $\mfg_0$-modules,  $V_{i_1}\otimes V_{i_2}\otimes\cdots \otimes V_{i_r}\cong V_1^{\otimes r-a}\otimes  V_0^a$ for some non-negative integer $a\le r$ with  $a=\# \{i_j\mid i_j=0\}$. The corresponding isomorphism is given by
acting  a unique  element $w$ on the right hand side of   $V_1^{\otimes r-a}\otimes  V_0^a$ , where $w$ is  a distinguished right coset representatives of $ \mathfrak S_a\times \mathfrak S_{r-a}$ in $\mathfrak S_r$.
By Lemma~\ref{liehwv}, all $\mfg_0$-highest weight vectors
of $V_{i_1}\otimes V_{i_2}\otimes\cdots \otimes V_{i_r}$ with highest weight $\tilde\l$ are $v_{\tilde\lambda }w_{\lambda^{(1)}} y_{\mu^{(1)}} w_{\lambda^{(2)}} x_{\mu^{(2)}}  d(\t_1) d(\t_2) w$ for all $\t_1\in \Std(\mu^{(1)})$ and $\t_2\in \Std(\mu^{(2)})$.  Therefore, the $\mathbb C$-space $V_{\tilde\l}$  of all $\mfg_0$-highest weight vectors of $V^{\otimes r}$ with highest weight $\tilde \l$  has  a basis $\{
  v_{\tilde\lambda } w_{\lambda^{(2)}} x_{\mu^{(2)}} w_{\lambda^{(1)}} y_{\mu^{(1)}} d(\t)\mid \t\in \Std(\lambda')\}$, where $\mu=\l'$.
  By \eqref{vpqt},  $\{ v_{\t}\mid \t\in \Std(\lambda')\}$ is $\mathbb C$-linear independent. Finally, since there is a  one to one correspondence between  $\mathfrak{g}$-highest weight vectors of $M_{pq}^{r0}  $  and  $\mathfrak{g}_{ 0}$-highest weight vectors of $V^{\otimes r}$ (cf. \cite[Lemmas~5.1-5.2]{SHK}),  and $\dim V_{\tilde \l}=\#  \{ v_{\t}\mid \t\in \Std(\lambda')\}$, one obtains that
   $\{ v_{\t}\mid \t\in \Std(\lambda')\}$ is a basis of $V_{\bar \l}$. \end{proof}

In the remaining part of this section, we want to establish the relationship between $V_{\bar \l}$ with
a special  cell module of $\mathscr H_{2, r}$ with respect to $\lambda\in \Lambda_2^+(r)$.
 This result will be needed in section~6.  We go on using $-x_1$ instead of $x_1$ in \cite{BS4}. In this case, the current  $-p$ and $m-q$ are the same as $p$ and $q$ in \cite{BS4}.

\begin{proposition}\label{cell-iso1} For any $\lambda\in \Lambda_2^+(r)$, $V_{\bar \lambda}\cong\mathfrak x_\lambda w_\lambda \mathfrak y_{\lambda'}\mathscr H_{2, r}$  as right $\mathscr H_{2, r}$-modules, where $V_{\bar\l}$ is defined in Theorem~$\rm\text{\rm\ref{hecg}}$.
 \end{proposition}

\begin{proof} By Lemma~\ref{ineq}\,(2),
 $S^\lambda:=\mathfrak x_\lambda w_\lambda \mathfrak y_{\lambda'}\mathscr H_{2, r}$ has a basis $M=\{\mathfrak x_\l w_\l \mathfrak y_{\l'} d(\t)\mid \t\in \Std(\lambda')\}$.
It follows from Theorem~\ref{hecg} that  there is a linear isomorphism $\phi: V_{\bar\lambda}\rightarrow S^\lambda$  sending $v_\t$ to $\mathfrak x_\l w_\l \mathfrak y_{\l'} d(\t)$.
 Obviously, $\phi$ is a  right $\mathfrak S_r$-homomorphism. In order to show that $\phi$ is a right $\mathscr H_{2, r}$-homomorphism,  it suffices to prove that \begin{equation} \label{xkc} \phi(v_\t x_k)=\phi(v_\t) x_k \text{ for $1\le k\le r$.}\end{equation}
Denote $a=|\lambda^{(1)}|$. If $1\le k\le r-a$, then $\tilde \pi_{\l'}  x_k=\pi_{r-a} (m-q)x_k=\pi_{r-a} (m-q)(-p-L_k)$. Since  $\phi$ is a right $\mathfrak S_r$-homomorphism,  \eqref{xkc} holds for $1\le k\le r-a$.
If $r-a+1\le k\le r$, then $x_k=s_{k, r-a+1}  x_{r-a+1} s_{r-a+1, k}  -\sum_{j=r-a+1}^{k-1} (j, k)$.
By Lemma~\ref{ineq}\,(1),   \begin{equation} \label{spec1} \pi_\l  w_a \tilde \pi_{\l'} x_k=\pi_\l w_a \tilde \pi_{\l'}   \Big(-p-\mbox{$\sum\limits_{j=r-a+1}^{k-1}$}  (j, k)\Big). \end{equation}
On the other hand, $\tilde \pi_{\l'}  x_k=x_k \tilde \pi_{\l'} $ and $v_{\tilde \lambda} \otimes v_{pq}  w_{\lambda^{(1)}} w_{\lambda^{(2)}} y_{\mu^{(1)}} x_{\mu^{(2)}} w_a$ is a linear combination of elements
$ v_{\mathbf i}\otimes v_{pq}$, for some $\mathbf i\in I(m|n, r)$ such that  $v_{i_j}\in V_0$ for all $r-a+1\le j\le r$.
By Lemma~\ref{xsfor}\,(1), $x_k$ acts on $ v_{\mathbf i}\otimes v_{pq}$ as $-p-L_k$. In order to verify  \eqref{xkc} for $k\ge r-a+1$,  by \eqref{spec1}, it remains to show that
\begin{equation}\label{ik1} v_{\mathbf i}\otimes  v_{pq}      (i, k)\tilde \pi_{\l'}=0
\ \ \text{for all $i, 1\le i\le r-a$.}\end{equation} Write $v_{\mathbf i}\otimes  v_{pq}      (i, k)=v_{\mathbf j}$ up to a sign. Then $v_{j_i}\in V_0$ and  $v_{\mathbf j} (1, i)\tilde \pi_{\l'}=0$ by  Lemma~\ref{xsfor}\,(1).
 Since  $(1, i)\tilde \pi_{\l'}= \tilde\pi_{\l'}(1, i)$, and $(1, i)$ is invertible,   $v_{\mathbf j}  \tilde \pi_{\l'}=0$, proving \eqref{ik1}.
\end{proof}

\begin{coro}\label{heck-kac0}
Suppose  $\lambda\in \Lambda_2^+(r)$.  As right $\mathscr H_{2, r}$-modules,
\begin{equation}\label{kac-cell23} \Hom_{U(\mfg)} (K_{\bar \lambda}, M_{pq}^{r0} ) \cong \tilde \Delta(\lambda')\end{equation}  where $\tilde \Delta(\lambda')$ is the right cell module  defined via the cellular basis of $\mathscr H_{2, r}$ in Lemma~$\text{\rm\ref{cell-h2}\,(2)}.$
\end{coro}

\begin{proof} For any $\mfg$-highest weight vector $v$ of $M_{pq}^{r0} $
 with highest weight $\bar \lambda$, there is a unique $U(\mfg)$-homomorphism   $f_v: K_{\bar \lambda} \rightarrow  \mathbf U(\mfg) v\subset M_{pq}^{r0}$
 sending $v_{\bar \lambda}$ to $v$, 
 where $v_{\bar  \lambda}$ is the highest weight vector of $K_{\bar  \lambda}$.
 Further $f_v$ can be considered as a homomorphism in $\Hom_{U(\mfg)} (K_{\bar \lambda}, M_{pq}^{r0} )$ by composing an embedding homomorphism.

 For any $0\neq f\in \Hom_{U(\mfg)} (K_{\bar\l}, M_{pq}^{r0})$, $f(v_{\bar\l})$ is a highest
 weight vector of $ M_{pq}^{r0}$. By Theorem~\ref{hecg},  $f(v_{\bar\l})$ is a linear combination of
 $v_\t$'s, for $\t\in \Std(\lambda')$. So, $f$ can be written as a linear combination of $f_{v_\t}$'s.
 Thus, $\{f_{v_\t}\mid \t\in \Std(\lambda')\}$ is a basis of  $\Hom_{U(\mfg)} (K_{\bar \lambda}, M_{pq}^{r0}) $.
 Let $V_{\bar\l}$ be defined in Theorem~\ref{hecg}.  Then the linear isomorphism $\phi:  \Hom_{U(\mfg)} (K_{\bar \lambda}, M_{pq}^{r0})\rightarrow V_{\bar\l}$ sending $f_{v_\t}$ to $v_\t$ for any $\t\in \Std(\lambda')$
 is a right $\mathscr H_{2, r}$-homomorphism. By  Lemma~\ref{ineq}\,(5) and
  Proposition~\ref{cell-iso1}, $V_{\bar\l}\cong \tilde \Delta(\lambda')$, proving \eqref{kac-cell23}.
 \end{proof}

In the remaining part of this section, we always assume   $p-q\le -m$. If $p-q\ge n$, one can switch roles between $p$ and $q$ (or consider the dual module of $M_{pq}^{r0}$).  Without loss of any generality,  we assume $p, q\in \mathbb Z$.

Let $\l\in\L^+_{2}(r)$ with $r\le\min\{m,n\}$. Then $\l$ corresponds to a dominant weight $\bar\l$ defined in \eqref{barl}.
 In particular, $\bar\emptyset=\l_{pq}$.
Following \cite{BS4, GS, SZ3}, we are going to represent a dominant weight $\bar\l$ in a unique way by a {weight diagram} $D_\l$.
First we write (cf. \eqref{lrho})
\equa{rho-lambda=}{\bar\l^\rho\!=\!\bar\l\!+\!\rho\!=\!
(\bar\l^{L,\rho}_1,...,\bar\l^{L,\rho}_m\,|\,{\bar\l^{R,\rho}_1},...,{\bar\l^{R,\rho}_n}).}
Denote
\equan{set-s-l}{
\begin{aligned}
S(\l)\LE&=\{\bar\l^{L,\rho}_i\,|\,i=1,...,m\},&\quad
S(\l)\RI&=\{{-\bar\l^{R,\rho}_j}\,|\,j=1,...,n\},\\
S(\l)&=S(\l)\LE\cup S(\l)\RI,
&\quad  S(\l)\BO&=S(\l)\LE\cap S(\l)\RI.
\end{aligned}}
\begin{definition}\label{weight-d}The {\it weight diagram} $D_\l$ associated with the dominant weight $\bar\l$ is a line with vertices indexed by $\Z$ such that each  vertex $i$
is associated with a symbol $D_\l^{\,i}=\emptyset,<,>$ or $\times$
according to whether $i\notin S(\l)$, $i\in S(\l)\RI\setminus
S(\l)\BO$, $i\in S(\l)\LE\setminus S(\l)\BO$ or $i\in S(\l)\BO$.
\end{definition}
For example, if $p,q\in\Z$ with $p\le q-m$, then
the weight diagram $D_{\emptyset}$ of $\bar\emptyset=\l_{pq}$ is given by
\equa{Diagram-l}{
\raisebox{-0pt}\ \ \ \ .\ .\ .\
\line(1,0){12}\raisebox{-5pt}{$\stackrel{\dis\empty}{\footnotesize }$}\line(1,0){12} \raisebox{-5pt}{$\stackrel{\dis
>}{\circ}$}\put(-15,-9){$\ssc p-m+1$}\line(1,0){12} \raisebox{-5pt}{$\stackrel{\dis
>}{\circ}$}\line(1,0){12}\ldots
\line(1,0){12}\raisebox{-5pt}{$\stackrel{\dis >}{\circ}$}\put(-8,-9){$\ssc p$}
\line(1,0){12}\raisebox{-5pt}{$\stackrel{\dis \empty}{\circ}$}
\line(1,0){12}\raisebox{-5pt}{$\circ$}\line(1,0){12}\ldots
\line(1,0){12}\raisebox{-5pt}{$\stackrel{\dis <}{\circ}$}\put(-15,-9){$\ssc q-m+1$}
\line(1,0){12}\raisebox{-5pt}{$\stackrel{\dis <}{\circ}$}\line(1,0){12}\ldots
\line(1,0){12}\raisebox{-5pt}{$\stackrel{\dis <}{\circ}$}
\line(1,0){12}\raisebox{-5pt}{$\stackrel{\dis <}{\circ}$}\put(-15,-9){$\ssc q-m+n$}
\line(1,0){12}\raisebox{-5pt}{$\circ$}\line(1,0){12} \ .\ .\ .,
}
where, for simplicity, we have associated vertex $i$ with nothing if
$D_\l^i=\emptyset$.  Note that $\sharp S(\emptyset)\BO=0$, i.e., $\l_{pq}$ is typical.

\begin{definition}\label{topp}Let $ \bar \lambda$ be as in \eqref{barl}, where  $\l\in\L^+_2(r)$. \begin{itemize}\item[(1)]
Let $\bar\l^{\rm top}$ be the unique dominant weight such that $L_{\bar \l}$ is the simple submodule of the Kac-module $K_{\bar\l^{\rm top}}$.
Then $\bar\l^{\rm top}$ is obtained from $\bar\l$ via the unique longest right path (cf.~\cite[Definition 5.2]{SZ3}, ~\cite[Conjecture~4.4]{vZ})   or via
a raising operator (cf.~\cite{B}).
For example,
if $D_\l$ is given by
\equa{e1-Diagram-l}{
\raisebox{-0pt}\ \ \ \ .\ .\ .\
\line(1,0){12}\raisebox{-5pt}{$\stackrel{\dis
\empty}{\,0\,}$}\line(1,0){12} \raisebox{-5pt}{$\stackrel{\dis
\times}{\,1\,}$}\line(1,0){12} \raisebox{-5pt}{$\stackrel{\dis
\times}{\,2\,}$}\line(1,0){12}\raisebox{-5pt}{$\,3\,$}
\line(1,0){12}\raisebox{-5pt}{$\stackrel{\dis \times}{\,4\,}$}
\line(1,0){12}\raisebox{-5pt}{$\stackrel{\dis >}{\,5\,}$}
\line(1,0){12}\raisebox{-5pt}{$\,6\,$}
\line(1,0){12}\raisebox{-5pt}{$\stackrel{\dis \times}{\,7\,}$}
\line(1,0){12}\raisebox{-5pt}{$\stackrel{\dis <}{\,8\,}$}
\line(1,0){12}\raisebox{-5pt}{$\,9\,$}
\line(1,0){12}\raisebox{-5pt}{$\stackrel{\dis <}{\,10\,}$}
\line(1,0){12}\raisebox{-5pt}{$\,11\,$}\line(1,0){12} \ .\ .\ .,
}%
then  the weight diagram $D_{\lambda^{\rm top}}$ of $\bar\lambda^{\rm top}$
 is given by
\equa{Te1-Diagram-l}{
\raisebox{-0pt}\ \ \ \ .\ .\ .\
\line(1,0){12}\raisebox{-5pt}{$\stackrel{\dis
\empty}{\,0\,}$}\line(1,0){12} \raisebox{-5pt}{$\stackrel{\dis
\empty}{\,1\,}$}\put(-5,23){$\line(0,-1){11}$}\put(-5,23){$\line(1,0){57}$}
\line(1,0){12} \raisebox{-5pt}{$\stackrel{\dis
\empty}{\,2\,}$}\put(-5,19){$\line(0,-1){7}$}\put(-5,19){$\line(1,0){17}$}
\line(1,0){12}\raisebox{-5pt}{$\stackrel{\dis
\times}{\,3\,}$}\put(-5,19){$\vector(0,-1){7}$}\put(-5,19){$\line(-1,0){17}$}
\line(1,0){12}\raisebox{-5pt}{$\stackrel{\dis \empty}{\,4\,}$}\put(-5,19){$\line(0,-1){7}$}\put(-5,19){$\line(1,0){37}$}
\line(1,0){12}\raisebox{-5pt}{$\stackrel{\dis >}{\,5\,}$}
\line(1,0){12}\raisebox{-5pt}{$\stackrel{\dis
\times}{\,6\,}$}\put(-5,19){$\vector(0,-1){7}$}\put(-5,19){$\line(-1,0){27}$}
\line(1,0){12}\raisebox{-5pt}{$\stackrel{\dis \empty}{\,7\,}$}\put(-5,19){$\line(0,-1){7}$}\put(-5,19){$\line(1,0){37}$}
\line(1,0){12}\raisebox{-5pt}{$\stackrel{\dis <}{\,8\,}$}
\line(1,0){12}\raisebox{-5pt}{$\stackrel{\dis
\times}{\,9\,}$}\put(-5,19){$\vector(0,-1){7}$}\put(-5,19){$\line(-1,0){27}$}
\line(1,0){12}\raisebox{-5pt}{$\stackrel{\dis <}{\,10\,}$}
\line(1,0){12}\raisebox{-5pt}{$\stackrel{\dis
\times}{\,11\,}$}\put(-7,23){$\vector(0,-1){11}$}\put(-7,23){$\line(-1,0){177}$}
\line(1,0){12} \ .\ .\ .,
}%
where the $\times$'s at vertices $9,6,3,11$ in \eqref{Te1-Diagram-l} are respectively obtained from the
$\times$'s at vertices $7,4,2,1$ in \eqref{e1-Diagram-l} (thus every symbol ``\,$\times$\,''
 is always moved to the unique empty place at its right side which is closest to it, under the rule that the rightmost ``\,$\times$\,'' should be moved first, as indicated in \eqref{Te1-Diagram-l}).
 Alternatively,
 $\bar\l$ is obtained from $\bar\l^{\rm top}$ via the unique longest left path.
\item[(2)]Let $\l^{\rm top}\in\L_{2}^+(r)$ be the unique bipartition such that $\bar\l^{\rm top}=\l_{pq}+\tilde\l^{\rm top}$ (cf.~\eqref{tilla} and \eqref{barl}).
\end{itemize}
\end{definition}

Write  $p=q-m-k$ for some  $k\in\N$. 
If  $\mu=((\mu^L_1,...,\mu^L_m),(\mu^R_1,...,\mu^R_n))\in\L_{2}^+(r)$, then  $\mu'$ is Kleshchev with respect to $u_1=-p,$ $u_2=m-q$ (cf.~\eqref{Kle}) if and only if
\equa{kle-1}{\mu_i^{L}\ge\mu_i^{R}-k\mbox{ \ for all possible }i.}
Following \cite[IV]{BS4}, we denote $I^+_{pq}=\{p-m+1,p-m+2,...,q-m+n\}$. For any  $\l\in\L_{2}^+(r)$ and  any $j\in I^+_{pq}$, set
\begin{eqnarray}
\label{S-1+}
&&I^{\emptyset}_{\ge j}(\l)=\Z^{\ge j}\cap(I^+_{pq}\setminus S(\l)\cap I^+_{pq}),\\
\label{S-1-}
&&I^{\emptyset}_{\le j}(\l)=\Z^{\le j}\cap(I^+_{pq}\setminus S(\l)\cap I^+_{pq}),\\
\label{St-1+}
&&I^{\times}_{\ge j}(\l)=\Z^{\ge j}\cap(I^+_{pq}\cap S(\l)\BO,\\
\label{St-1-}
&&I^{\times}_{\le j}(\l)=\Z^{\le j}\cap(I^+_{pq}\cap S(\l)\BO.
\end{eqnarray}
In terms of the above notations, Brundan and Stroppel
 \cite[IV, Lemma 2.6]{BS4} have proved that the indecomposable tilting module $T_{\bar\l}$ is a direct summand of
$M_{pq}^{r0}$ if
\equa{1111-1}{S(\l)\subset I_{pq}^+\mbox{ \ and \ } \#I^{\emptyset}_{\ge j}(\l)\ge\#I^{\times}_{\ge j}(\l)\mbox{ for all }j\in I^+_{pq}.}
These two conditions on  bipartition $\l$ (or weight $\bar\l$) are equivalent to the following conditions on $\l^{\rm top}$ (which can be seen from \eqref{e1-Diagram-l}--\eqref{Te1-Diagram-l} in case $I^+_{pq}=\{1,2,...,11\}$):
\equa{1111-1-top}{S(\l^{\rm top})\subset I_{pq}^+\mbox{ \ and \ } \#I^{\emptyset}_{\le j}(\l^{\rm top})\ge\#I^{\times}_{\le j}(\l^{\rm top})\mbox{ for all }j\in I^+_{pq}.}

\begin{lemma}\label{kle-bi}Let $\mu\in\L_{2}^+(r)$ such that $\mu'$ is Kleshchev with respect to $u_1=-p,u_2=m-q$, where $p=q-m-k$ with $k\in\N$. 
Then
\equa{mu-1111-1-top}{S(\mu)\subset I_{pq}^+\mbox{ \ and \ } \#I^{\emptyset}_{\le j}(\mu)\ge\#I^{\times}_{\le j}(\mu)\mbox{ for all }j\in I^+_{pq}.}
\end{lemma}
\begin{proof}
We have (cf. \eqref{lrho})
\equa{l-pq}{\l_{pq}+\rho
=(q-m-k,...,q-2m-k+1\,|\,-q+m-1,...,-q+m-n).}
Thus for $i=1,...,m,$ we have (cf. \eqref{rho-lambda=})   $\bar\mu^{L,\rho}_i=\mu_i+q-m-k\ge q-2m-k+1$ and $\bar\mu^{L,\rho}_i\le q+n-m$ (as $\mu_i\le r\le n$), i.e., $\bar\mu^{L,\rho}_i\in I^+_{pq}$. Similarly,
$-\bar\mu^{R,\rho}_j\in I^+_{pq}$ for $j=1,...,n$. Hence, $S(\mu)\subset I^+_{pq}$.

To prove the other assertion of \eqref{mu-1111-1-top}, note that the weight diagram $D_\mu$ of $\bar\mu$ is obtained from $D_{\emptyset}$ (cf.~\eqref{Diagram-l}) by moving the ``\,$>$\,'' at vertex $p-i$ for all $i$ with $0\le i\le m-1$ to its right side to vertex $p-i+\mu_{i+1}^{L}$, and moving the ``\,$<$\,'' at vertex $q-m+j$ for all $j$ with $1\le j\le n$ to its left side to vertex $q-m+j-\mu^R_j$ (if ``\,$<$\,'' meets
``\,$>$\,''  at the destination vertex, then two symbols ``\,$<$\,'' and ``\,$>$\,'' are combined to become the symbol ``\,$\times$\,''). Since $\mu'$ is Kleshchev, condition \eqref{kle-1}
shows that in order to produce a ``\,$\times$\,'' at some vertex $i$ of $D_\mu$, a ``\,$>$\,'' at some vertex $j$ with $j<i$ must be moved to vertex $i$, i.e., an ``\,$\emptyset$\,'' must appear in some vertex $j'$ with $j'\le j<i$, i.e.,
 \eqref{1111-1-top} holds.\end{proof}

\begin{coro}\label{Kl-heck-kac} Suppose  $\lambda\in \Lambda_2^+(r)$ such that $(\l^{\rm top})'$ is Kleshchev,  where
$(\lambda^{\rm top})'$ is the conjugate of $\l^{\rm top}\in\L_{2}^+(r)$. Then $T_{\bar\l}$ is a direct summand of $M_{pq}^{r0}$. Further, any indecomposable direct summand of  $M_{pq}^{r0}$
is of form  $T_{\bar\l}$ for some $\lambda\in \Lambda_2^+(r)$ such that $(\l^{\rm top})'$ is Kleshchev.
\end{coro}

\begin{proof}The first assertion follows from \cite[IV, Lemma 2.6]{BS4} and Lemma~\ref{kle-bi}. To prove the last assertion, since  $r\le \min\{m, n\}$, by Theorem~\ref{bshec}, $\text{End}_{U(\mathfrak {gl}_{m|n})} (M_{pq}^{r0})\cong \mathscr H_{2, r}$. So, the number of non-isomorphic indecomposable direct summands of
$\mathfrak{gl}_{m|n}$-module $M_{pq}^{r0}$ is equal to that of non-isomorphic irreducible $\mathscr H_{2, r}$-modules, which is equal to the number of Kleshchev bipartitions in $\Lambda_{2}^+(r)$. Now, everything is clear.
\end{proof}

\begin{coro}\label{heck-kac} Suppose  $\lambda\in \Lambda_2^+(r)$ such that $(\l^{\rm top})'$ is Kleshchev.  As right $\mathscr H_{2, r}$-modules,
\begin{equation}\label{kac-cell2} \Hom_{U(\mfg)} (T_{\bar\lambda}, M_{pq}^{r0} ) \cong  P^{(\lambda^{\rm top})'}, \end{equation} where $ P^{(\lambda^{\rm top})'}$ is the projective cover of
 $D^{(\lambda^{\rm top})'}$    which is the simple head of $\tilde \Delta((\lambda^{\rm top})')$.
\end{coro}

\begin{proof} Since  $r\le \min\{m, n\}$ and  $(\lambda^{\rm top})'$ is Kleshchev, by  Corollary~\ref{Kl-heck-kac},
  $T_{\bar \lambda}$ is  a direct summand of $M_{pq}^{r0}$, forcing  $0\neq \Hom_{U(\mfg)} (T_{\bar\lambda}, M_{pq}^{r0}) $ to be a direct summand of $\mathscr H_{2, r}$.
  We claim that $\Hom_{U(\mfg)} (T_{\bar\lambda}, M_{pq}^{r0}) $ is indecomposable. If not, then the number of indecomposable direct summands of the right $\mathscr H_{2, r}$-module $\mathscr H_{2,r}$ is strictly bigger than $\sum_{\bar \lambda} \ell_{\bar \lambda}$ if  we write $M_{pq}^{r0}$ as
 $M_{pq}^{r0}=\oplus_{\bar \lambda} T_{\bar \lambda}^{\oplus \ell_{\bar\lambda}}$ with $\ell_{\bar\l}\neq 0$.

 On the other hand, since $M_{pq}^{r0}$ is a right $\mathscr H_{2,r}$-module, we can consider the right exact  functor $\mathfrak F:=M_{pq}^{r0}\otimes_{\mathscr H_{2,r}}?$ from  the category of  left $\mathscr H_{2,r}$-modules
 to the category of left $U(\mfg)$-modules.  We have an epimorphism from $\mathfrak F(P^\mu)$ to  $\mathfrak F(\tilde \Delta(\mu))$, where $P^\mu$ is  any principal indecomposable left $\mathscr H_{2, r}$-module and
 $ \tilde \Delta(\mu)$ temporally denotes the left cell module of $\mathscr H_{2, r}$ defined via the cellular basis of $\mathscr H_{2,r}$ in Lemma~\ref{cell-h2}\,(1) with the simple head $D^\mu$.  By Lemma~\ref{ineq}\,(5) and Theorem~\ref{bshec},  $\mathfrak F(\tilde \Delta(\mu))\neq 0$, forcing   $\mathfrak F(P^\mu)\neq 0$. So,  $\mathfrak F(P^\mu)$ is a direct sum of
 indecomposable direct summand of $U(\mfg)$-module $M_{pq}^{r0}$. In particular, $\sum_{\bar \lambda} \ell_{\bar \lambda}$ is no less than
 the number of  indecomposable direct summands of left $\mathscr H_{2,r}$-module $\mathscr H_{2, r}$. This is a contradiction since the number of
 indecomposable direct summands of left $\mathscr H_{2, r}$-module $\mathscr H_{2, r}$ is equal to that of
 indecomposable direct summands of right $\mathscr H_{2, r}$-module $\mathscr H_{2, r}$. So, $\mathfrak F(T_{\bar \lambda})$ is a principal  indecomposable right $\mathscr H_{2,r}$-module.
Since  $K_{\l^{\rm top}}\hookrightarrow T_{\bar \l}$,  $\Hom_{U(\mfg)} (T_{\bar\lambda}, M_{pq}^{r0} )\twoheadrightarrow \Hom_{U(\mfg)} (K_{\bar\lambda^{\rm top}}, M_{pq}^{r0})$. By Corollary~\ref{heck-kac0},
$\Hom_{U(\mfg)} (K_{\bar\lambda^{\rm top}}, M_{pq}^{r0})\cong \tilde \Delta((\lambda^{\rm top})') $.  Since $\Hom_{U(\mfg)} (T_{\bar\lambda}, M_{pq}^{r0} )$ is a principal indecomposable right $\mathscr H_{2, r}$-module,
it implies that $\tilde \Delta((\lambda^{\rm top})')$ has the unique simple head, denoted by $D^{(\l^{ \rm top})'}$. Thus,
 $\Hom_{\mfg} (T_{\bar\lambda}, M_{pq}^{r0}) \cong  P^{(\lambda^{\rm top})'}$.
\end{proof}

Brundan-Stroppel have already proved that  decomposition numbers of $\mathscr H_{2, r}$ arising from super Schur--Weyl duality in \cite{BS4}  can be determined by the multiplicity of Kac-modules in indecomposable tilting modules appearing in $M_{pq}^{r0}$. This result can also be seen via the exact functor  $\Hom_{U(\mfg)} (?, M_{pq}^{r0})$.

\def\wt{{\rm wt}_\l}

\section{Highest weight vectors in $M_{pq}^{rt}$  }\label{hvw}
In this section, we classify $\mfg$-highest weight vectors
of  $\mathfrak{gl}_{m|n}$-module  $M_{pq}^{rt}$  over $\mathbb C$. As an application, we set up explicit relationship between Kac (resp. indecomposable
tilting)  modules of $\mfg$ and cell (resp. principal indecomposable) modules of $\DBr$. This gives us  an efficient way to calculate decomposition numbers of $\mathscr B_{2, r, t}$.
Throughout, assume $r, t\in \mathbb Z^{>0}$ such that $r+t\le \min\{m, n\}$.
The case  $t=0$ has been dealt with  in section~5. By symmetry,  one can also classify highest weight  vectors of $M_{pq}^{0t}$ via those in section~5.
 The following    result, which  is  the counterpart of Lemma~\ref{xsfor},   can be verified directly.

\begin{lemma} \label{xsfor1} Suppose $\mathbf i\in I(m|n, r)$, ${\mathbf j}\in \bar I(m|n, t)$ $($cf. \eqref{imn}$)$ and $1\le k\le t$.
\begin{enumerate} \item
$v_{\mathbf i}\otimes v_{pq}\otimes \bar v_{\mathbf j}  \bar x_k'= q v_{\mathbf i}\otimes v_{pq}\otimes \bar v_{\mathbf j} $  if $1+m\le j_k\le m+n$.
\item $v_{\mathbf i}\otimes v_{pq} \otimes\bar v_{\mathbf j} \bar  x_k'=p v_{\mathbf i}\otimes v_{pq} \otimes\bar v_{\mathbf j}  +\sum_{j=m+1}^{m+n} (-1)^{\sum_{l=1}^{k-1} [j_l]} v_{\mathbf i} \otimes (E_{j j_k}v_{pq})\otimes \bar v_{\mathbf \ell} $
if $1\le j_k\le m$, where   $\mathbf \ell\in \bar I(m|n, t)$ which is obtained from $\mathbf j$ by using $j$ instead of $j_k$ in $\mathbf j$. In particular, the weight of $\bar v_{\mathbf \ell}$ is strictly bigger  than that of $\bar v_{\mathbf j}$.
 \end{enumerate}
\end{lemma}

 For any integral weight $\xi=(\xi_1,...,\xi_m\,|\,\xi_{m+1},...,\xi_{m+n})$ of $\mfg$, let
 $$\xi^L=(\xi_1^L,...,\xi_m^L)=(\xi_1,...,\xi_m), \text{ and $\xi^R=(\xi_1^R,...,\xi_m^R)=(\,\xi_{m+1},...,\xi_{m+n})$.}$$
  We define   two bicompositions $\mu, \nu$ such that all $\mu^{(1)}_i,\mu^{(2)}_j,\nu^{(1)}_i,\nu^{(2)}_j$ are zero except that
\begin{enumerate}\item for $1\le i\le m$, $\mu^{(1)}_i=\xi^L_i$ if $\xi^L_i>0$ or  $\nu^{(1)}_{m-i+1}=-\xi^L_i$ if $\xi^L_i<0$.
\item for $1\le j\le n$, $\mu^{(2)}_{j}=\xi^R_{j}$ if  $\xi^R_{j}>0$ or  $\nu^{(2)}_{n-j+1}=-\xi^R_{j}$ if $\xi^R_{j}<0$.
\end{enumerate}
Then both $\mu$ and $\nu$ correspond to integral  weights of $\mfg$.
In particular, $\xi=\mu-\hat\nu$ with \equa{hat-nu}{\hat {\nu}=(\nu_m^{(1)},...,\nu^{(1)}_1\,|\,\nu^{(2)}_{n},...,\nu^{(2)}_{1})\in\fh^*.}
Conversely, if $\mu$ and $\nu$ are two bicompositions, then  $\xi=\mu-\hat \nu$ is a integral weight of $\mfg$.
For instance, if $\xi=(r-4,1, 0,\cdots,0, -1,-(t-5)\mid 2,1, 0, \cdots, 0, -1, -3)$, then  $\mu=((r-4,1), (2,1))$ and $\nu=((t-5,1), (3,1))$ such that $\xi=\mu-\hat \nu$.

\begin{definition}\label{wt-l}
For any  $\l=(f, \mu, \nu)\in \Lambda_{2, r, t}$, let  $\bar\l:=\l_{pq}+\mu-\hat\nu $ and  $\tilde\l:=\mu-\hat\nu$.
Since $r+t\le \min\{m, n\}$, both $\mu$ and $\nu$ correspond to integral  weights of $\mfg$ as above such that
\begin{equation}\label{COn11}
\mbox{$\mu_i\nu_{m+1-i}=0$ \ for \ $1\le i\le m$ \ and \ $\mu_{m+j}\nu_{m+n+1-j}=0$ \ for \ $1\le j\le n$ , }\end{equation}
\end{definition}

\begin{lemma}\label{lemm111}
For any $\mfg$-highest weight $\Lambda$ of $M_{pq}^{rt}$, there is a unique triple $\l=(f, \mu,\nu)\in \Lambda_{2, r, t}$ such that $\Lambda=\bar\l$.
\end{lemma}
\begin{proof}
By \cite[Lemma 5.20]{RSu}, $\Lambda=\l_{pq}+\eta-\zeta$ for some bicompositions $\eta$ and $\zeta$ of sizes $r$ and $t$ respectively. For $i\in I$, let $\xi_i=\min\{\eta_i,\zeta_i\}$ and $f=\sum_{i\in I}\xi_i$. Then we obtain a weight $\xi$, and two bicompositions  $\mu:=\eta-\xi$ and $\gamma:=\zeta-\xi$ such that $|\mu|=r-f$, $|\gamma|=t-f$ and  $\Lambda=\l_{pq}+\mu-\gamma$.
Set $\nu=\hat\gamma$, then $\Lambda=\bar \l$ and \eqref{COn11} is satisfied by definition of $\xi$. Since  $\Lambda$ is dominant, $\mu,\nu$ must be bipartitions. Thus $\Lambda$ corresponds to $\l=(f, \mu,\nu)\in \Lambda_{2, r, t}$. Such a  $\l$ is unique.
\end{proof}

\begin{definition}\label{mixh} For each  $\lambda=(f, \mu, \nu)\in \Lambda_{2, r, t}$, denote
    $v_{\lambda}=v_{\mathbf i}\otimes v_{pq}\otimes v_{\mathbf j}$, where
 $$\mathbf i=(\mathbf i_{\mu^{(1)}}, \mathbf i_{\mu^{(2)}},\underset{f} {\underbrace{ 1, \cdots, 1}})\in I(m|n, r), \text{ and $\mathbf j=(\mathbf j_{\nu^{(2)}}, \mathbf j_{\nu^{(1)}},\underset{f}{\underbrace{ 1, \cdots, 1}})\in \bar I(m|n, t)$}\vspace*{-7pt}, $$ such that
 \begin{enumerate} \item
 $\mathbf j_{\nu^{(2)}}$  is obtained from $\mathbf i_{\nu^{(2)}}$  by using $m+n-i+1$ instead of $i$ for $1\le i\le n$,
 \item  $\mathbf j_{\nu^{(1)}}$   is obtained from $\mathbf i_{\nu^{(1)}}$ by using $m-i+1$ instead of $i$ for $1\le i\le m$.
 \end{enumerate}
\end{definition}
For instance, if $\l=(1, \mu, \nu)\in \Lambda_{2, 8, 10} $ with $\mu=((3, 1),(2,1))$ and $\nu=((4,1), (3,1))$, then
$\mathbf i=(1^3, 2, (m+1)^2, (m+2), 1) $ and $\mathbf j=((m+n)^3, (m+n-1), m^4, (m-1), 1)$. Thus,
$$v_\lambda=v_1\otimes v_{m+2}\otimes v_{m+1}^{\otimes 2}\otimes v_2\otimes v_1^{\otimes 3}\otimes v_{pq}\otimes \bar{v}_{m+n}^{\otimes 3}\otimes \bar{v}_{m+n-1}\otimes \bar{v}_m^{\otimes 4}\otimes\bar{v}_{m-1}\otimes \bar{v}_1.$$

\begin{definition} \label{rcs1} For any $(f, \mu, \nu)\in  \Lambda_{2, r, t}$, define \begin{enumerate}
\item  $w_{\mu, \nu} =w_\mu w_{\nu^{o}}$, with $\nu^{o}=(\nu^{(2)}, \nu^{(1)})$, $w_\mu=d(\ft_\mu)\in \mathfrak S_{r-f}$  and $w_{ \nu^{o}}=d(\ft_{ \nu^{o}})\in \bar {\mathfrak S}_{{t-f}}$,
\item
$v_{\lambda, \t, d, \kappa_d}=v_\l \mathfrak  e^f w_{\mu, \nu} \mathfrak y_{\mu'} \bar {\mathfrak y}_{(\nu^o)'}  d(\t) d x^{\kappa_d}$,  $\t\in \Std(\mu')\times \Std((\nu^{o})')$, $ d\in \mathcal D_{r,t}^f$ and $ \kappa_d\in \mathbf N_f$.
\end{enumerate}\end{definition}

\begin{theorem}\label{theorem-main1}
Suppose $r+t\le \min\{m, n\}$. \begin{enumerate} \item There is a bijection between the set of  dominant weights of $M_{pq}^{rt}$ and
$\Lambda_{2, r, t}$. \item If  $\l=(f, \mu, \nu)\in \Lambda_{2, r, t}$, then  $V_{\bar\l}$, the $\mathbb C$-space of all $\mfg$-highest weight vectors of $M_{pq}^{rt}$
with highest weight $\bar\l$,    has a basis $S:=\{v_{\lambda, \t, d, \kappa_d}\mid \t\in \Std(\mu')\times \Std((\nu^{o})'), d\in \mathcal D_{r,t}^f, \kappa_d\in \mathbf N_f\}$.\end{enumerate}
\end{theorem}

\begin{proof}
 Obviously, (1) follows from  Lemma~\ref{wt-l}. To obtain (2), we prove that for each $\l=(f, \mu, \nu)\in \Lambda_{2, r, t}$, $V_{\bar {\l}}$ has the required basis in the case either $f=0$ or $f>0$.

\medskip\textit{Case~1: $f=0$.}

By Definition~\ref{rcs1},
$v_{\lambda, \t, d, \kappa_d}=v_{\mathbf i}\otimes v_{pq} \otimes \bar v_{\mathbf j}
 w_{\mu} \mathfrak y_{\mu'} d(\t_1) w_{\nu^{o}}  \bar{\mathfrak y}_{(\nu^o)'}  d(\t_2)$,
where $\mathbf i$, $\mathbf j$ are defined in  Definition~\ref{mixh}. By Theorem~\ref{hecg},
$v_{\mathbf i}\otimes v_{pq} \otimes \bar v_{\mathbf j}
 w_{\mu} \mathfrak y_{\mu'} d(\t_1)$ can be regarded as  a $\mfg$-highest weight vector of $M_{pq}^{r0}$. Similarly,  $v_{\mathbf i}\otimes v_{pq} \otimes \bar v_{\mathbf j}w_{\nu^{o}} \bar {\mathfrak y}_{(\nu^o)'} d(\t_2)$ can be regarded as a $\mfg$-highest weight vector of $M_{pq}^{0t}$. Thus,
$v_{\lambda, \t, d, \kappa_d}$ is  a  $\mfg$-highest weight vector of $M_{pq}^{rt}$. The last assertion follows from
 arguments on counting the dimensions of $V_{\bar\l}$ and  that of  $\mfg_0$-highest weight vectors of $V^{rt}:=V^{\otimes r} \otimes W^{\otimes  t} $  with highest weight $\mu-\hat\nu$.

\medskip
\textit{Case~2: $f>0$.}

 For any $i\in I$,
 $$v_i\otimes \bar v_{i}e_{1}=(-1)^{[i]}\mbox{$\sum\limits_{j\in I}$} v_j\otimes \bar v_{j}.$$
Thus  $v_i\otimes \bar v_{i}e_{1}$ is unique up to a sign for different $i$'s.
  Since $M_{pq}^{rt}$ is a $(\mfg, \mathscr B_{2, r, t})$-bimodule,
we can switch $v_{i_{r-k}}$ and  $\bar v_{j_{t-k}}$  in $v_\lambda$ with $i_{r-k}=j_{t-k}$ to  $v_o$ and  $\bar v_o$  for  any fixed $o, 1\le o\le m+n$ simultaneously  when we consider
the action of  $E_{j,\ell}$  on $i_{r-k}$-th (resp. $j_{t-k}$-th) tensor factor of $v_{\lambda, \t, d, \kappa_d}$ for $0\le k\le f-1$.
Let
\begin{equation}\label{vtt} v_\t:=v_{i_{r-f}}\otimes \cdots v_{i_1}\otimes v_{pq}\otimes  \bar v_{j_{ 1}} \otimes \cdots \bar v_{j_{{t-f}}}  w_{\mu, \nu} x_{\alpha^{(2)}} y_{\alpha^{(1)}} y_{\beta^{(1)}}  x_{\beta^{(2)}} \pi_{r-f-a} (-p)\pi_{b}(q) d(\t),\end{equation}
where $\a^{(i)}$ (resp. $\b^{(i)}$) is the conjugate of $\mu^{(i)}$ (resp. $\nu^{(i)}$), $i=1,2$.   Applying Theorem~\ref{hecg} to both $V^{\otimes r-f}\otimes K_{\lambda_{pq}}$ and $K_{\lambda_{pq}}\otimes W^{\otimes t-f}$ yields $E_{ j,\ell} v_\t=0$. So,
  $E_{j,\ell} v_{\l, \t, d, \kappa_d}=0$ for any $j<\ell$.

 We claim  that $S$ is linear independent, where $S$ is given in (2). If so, each  $v_{\l, \t, d, \kappa_d}\neq 0$, forcing $v_{\l, \t, d, \kappa_d}$ to be a $\mfg$-highest weight vector of $M_{pq}^{rt}$ with highest weight $\bar \l$.

Suppose   $\mathbf i\in I(m|n, r_1-1)$ and $\mathbf j\in \bar I(m|n, t_1-1)$ with $r_1\le r$ and $t_1\le t$  such that  there are at least some  $k_0\in I_0$ and $\ell_0\in I_1$  satisfying $k_0, \ell_0\not\in \{ i_l,  j_o\}$ for all possible $i, o$'s.
We consider  $\sum_{k\in I} v_k\otimes v_{\mathbf i} \otimes v\otimes v_{\mathbf j}\otimes \bar v_k\in M_{pq}^{r_1, t_1}$, where $v\in B$ is  a basis element of  $K_{\l_{pq}}$ in \eqref{basis-k-l}.
Since $x_{r_1}'=x_{r_1}+L_{r_1}$ and $x_{r_1}'$ acts on $M_{pq}^{r_1, t_1}$ as $-\pi_{r_1, 0} (\Omega)$, where $\Omega$ is given in \eqref{def-Omega}, we have
$$\begin{aligned} \mbox{$\sum\limits_{k\in I}$} & v_k\otimes v_{\mathbf i} \otimes v\otimes v_{\mathbf j}\otimes \bar v_k  (x_{r_1}+L_{r_1})
=-\pi_{r_1, 0} (\Omega) \mbox{$\sum\limits_{k\in I}$}  v_k\otimes v_{\mathbf i} \otimes v\otimes v_{\mathbf j}\otimes \bar v_k\\
& =-\mbox{$\sum\limits_{k, i\in I}$} (-1)^{[k]+([k]+[i])([k]+[\mathbf i])} v_i\otimes v_{\mathbf i} \otimes E_{k,i} v\otimes \bar v_{\mathbf j} \otimes \bar v_{k},\\
 \end{aligned}$$
 where $[\mathbf i]=\sum_{j=1}^{r_1-1} [i_j]$.  So, up to some scalar $a$,  $\sum_{k=1}^{m+n}  v_k\otimes v_{\mathbf i} \otimes v\otimes v_{\mathbf j}\otimes \bar v_k  x_{r_1}$ contains the unique term
 $v_{k_0} \otimes v_{\mathbf i} \otimes v \otimes \bar v_{\mathbf j}\otimes  \bar v_{k_0}$.
 In particular, if $v\neq v_{pq}$,
 $\sum_{k\in I}  v_k\otimes v_{\mathbf i} \otimes v\otimes v_{\mathbf j}\otimes \bar v_k  x_{r_1}$
does not contribute terms with form $v_{k_0} \otimes v_{\mathbf i'} \otimes v_{pq} \otimes v_{\mathbf j'}\otimes \bar v_{k_0}$ for all possible $\mathbf i'$ and $\mathbf j'$. If $v=v_{pq}$, by Lemma~\ref{xsfor}, the previous scalar is $-p$. Similarly, the coefficient of $v_{\ell_0} \otimes v_{\mathbf i} \otimes v_{pq} \otimes v_{\mathbf j}\otimes \bar v_{\ell_0}$ in the expression of
$\sum_{k\in I}  v_k\otimes v_{\mathbf i} \otimes v\otimes v_{\mathbf j}\otimes \bar v_k  x_{r_1}$ is
$ -q$.  Assume
\begin{equation} \label{cd}  c\mbox{$\sum\limits_{k\in I}  v_k\otimes v_{\mathbf i} \otimes v_{pq}\otimes v_{\mathbf j}\otimes \bar v_k  x_{r_1}+d \sum\limits_{k\in I}$}  v_k\otimes v_{\mathbf i} \otimes v_{pq}\otimes v_{\mathbf j}\otimes \bar v_k=0\end{equation}
for some $c, d\in \mathbb C$. Then $d=cp=cq$ by  considering the coefficients of  $v_{k} \otimes v_{\mathbf i} \otimes v_{pq} \otimes v_{\mathbf j}\otimes \bar v_{k}$, $k\in \{k_0, \ell_0\}$ in the expression of
LHS of \eqref{cd}. If $c\neq 0$, then $p-q=0$. This is a contradiction since $\l_{pq}$ is typical in the sense of \eqref{Do-condi}. So, $c=d=0$  and hence $\sum_{k\in I}  v_k\otimes v_{\mathbf i} \otimes v_{pq}\otimes v_{\mathbf j}\otimes \bar v_k  x_{r_1}$ and $\sum_{k\in I}  v_k\otimes v_{\mathbf i} \otimes v_{pq}\otimes v_{\mathbf j}\otimes \bar v_k$ are linear independent.
Now, we assume \begin{equation}\label{r-lid} \mbox{$\sum\limits_{\t, d, \kappa_d}$} r_{\t, d,  \kappa_d} v_{\l, \t, d, \kappa_d}=0 \text{
for some  $  r_{\t, d,  \kappa_d}\in \mathbb C$.}\end{equation}
 We claim that $  r_{\t, d,  \kappa_d}=0$ for all possible $\t, d,  \kappa_d$. If not, then we  pick up a  $d\in \mathcal D_{r,t}^f$ such that
\begin{enumerate} \item  $r_{\t, d,  \kappa_d} \neq 0$, \item
 $d=s_{r-f+1,i_{r-f+1}}{\sc\!} \bar s_{t-f+1, j_{t-f+1}}{\sc\!}\! \cdots\!
s_{r,i_r}{\sc\!}\bar s_{t,{j_t}} $ and   $i_r>i_{r-1}>\cdots>i_{r-f+1}$,
\item $(i_r, \cdots, i_{r-f+1})$ is maximal with respect to lexicographic order. \end{enumerate}
Since $r+t\le \min\{m, n\}$ and $0<f\le \min\{r, t\}$, we can pick  $f$  pairs $(k_i, \ell_i)$, $r-f+1\le i\le r$  such that  \begin{enumerate} \item $k_i\in I_0$,   $\ell_i\in I_1$, $k_i>k_j$ and $\ell_i> \ell_j$ if $i> j$;
\item both  $v_{k_i}$ and $v_{\ell_i}$ are  not a tensor factor of $v_{\mathbf i_\mu}$,
\item   both  $\bar v_{k_i}$ and $\bar v_{\ell_i}$ are  not a tensor factor of  $\bar v_{\mathbf j}$. \end{enumerate}
We consider the terms $v_{\mathbf a} \otimes v_{pq}\otimes \bar v_{\mathbf b}$'s in
the expressions of  $v_{\l, \t, d, \kappa_d}$'s  in LHS of \eqref{r-lid} with  $r_{\t, d,  \kappa_d}\neq 0$
such that either  $v_{a_{i_h}}=v_{k_h}$ and $\bar v_{b_{i_{t-r+h}}}= \bar v_{k_h}$ or  $v_{a_{i_h}}=v_{\ell_h}$ and $\bar v_{b_{i_{t-r+h}}} =\bar v_{\ell_h}$
  for $r-f+1\le h\le r$.  Such terms occur in the expression of  $v_1^{\otimes f}\otimes\tilde  v_{\t} \otimes \bar v_1^{\otimes f} \mathfrak e^f d x^{\kappa_d}$, where $\tilde v_{\t}$ is a linear combination of  the terms
  in $v_{\t}$'s (cf. \eqref{vtt})  with forms $v_{\mathbf i'}\otimes v_{pq}\otimes \bar v_{\mathbf j'}$.
If $v_{a_h}=v_{k_h}$ and $\bar v_{b_{i_{t-r+h}}}= \bar v_{k_h}$, by previous arguments, the coefficient of $v_{\mathbf a}\otimes v_{pq}\otimes \bar v_{\mathbf b}$
in  $v_1^{\otimes f}\otimes v_{\t}\otimes \bar v_1^{\otimes f} e^f d x^{\kappa_d}$
 is $\prod_{h=r}^{r-f+1} (-p)^{\epsilon_h} $, where  $\epsilon_h=1$ if $\kappa_h=1$ and $0$ if $\kappa_h=0$.
 If $v_{a_h}=v_{\ell_h}$ and $\bar v_{b_{i_{t-r+h}}}= \bar v_{\ell_h}$, then the coefficient of $v_{\mathbf a}\otimes v_{pq}\otimes \bar v_{\mathbf b}$
in  $v_1^{\otimes f}\otimes\tilde  v_{\t}\otimes \bar v_1^{\otimes f} \mathfrak e^f d x^{\kappa_d}$  is $\prod_{h=r}^{r-f+1} (-q)^{\epsilon_h} $, where  $\epsilon_h=1$ if $\kappa_h=1$ and $0$ if $\kappa_h=0$.
 By \eqref{r-lid},  $\sum_\t r_{\t, d, \kappa_d} \tilde v_\t=0$ for any fixed $\kappa_d$. Thus, we can assume that $\kappa_d=(0,
 \cdots, 0)\in \mathbf N_f$. If we identify  $\tilde v_\t$ with its $v_{pq}$-component, then $\tilde v_\t$
   can be considered as $\mfg_0$-highest weight vectors of $V^{\otimes r-f}\otimes W^{\otimes  t-f}$ (cf. arguments in the proof of Theorem~\ref{hecg}) of the form
    $$\tilde v_\t=v_{i_{r-f}}\otimes \cdots v_{i_1}\otimes   \bar v_{j_{ 1}} \otimes \cdots \bar v_{j_{{t-f}}}  w_{\mu, \nu} x_{\alpha^{(2)}} y_{\alpha^{(1)}}\bar y_{\beta^{(1)}} \bar x_{\beta^{(2)}}d(\t).$$
So,  $r_{\t, d,\kappa_d}\!=\!0$, a contradiction.  This proves that  $S$ is $\mathbb C$-linear independent. Further,  $S$ is  a basis of $V_{\bar\l}$ since
the cardinality of $S$ is  $2^f |\mathcal D_{r,t}^f|\!\cdot\! |\Std(\mu')|\!\cdot\! |\Std(\nu')|$, which is   the dimension of space consisting of $\mfg_0$-highest weight vectors of
$V^{rt}$ with highest weight $\mu\!-\!\hat \nu$.
\end{proof}

\begin{definition}\label{f} Let $\FF={\rm Hom}_{U(\mfg)}(?,M_{pq}^{rt})$ be the functor from the category of finite dimensional left $\mfg$-modules to the category of right $\DBr$-modules over $\mathbb C$.\end{definition}

\begin{lemma}\label{exact} The functor $\FF$ is  exact.\end{lemma}

\begin{proof}Since  $\lambda_{pq}$ is typical, $M^{rt}$ is  projective, injective and tilting as left $\mfg$-module (e.g., \cite[IV]{BS4}). So, $\FF$ is exact.\end{proof}

\begin{proposition}\label{kac-ke} Suppose $\l=(f, \mu, \nu)\in \Lambda_{2, r, t}$.
Then $\FF(K_{\bar\l})\cong C(f, \mu', (\nu^o)')$, where $\nu^o=(\nu^{(2)}, \nu^{(1)})$.\end{proposition}
\begin{proof}
 By Proposition~\ref{classcell2}, there is  an explicit linear  isomorphism between  $C(f, \mu',  (\nu^{o})')$ and $V_{\bar\lambda}$, where $V_{\bar\lambda}$ is given in Theorem~\ref{theorem-main1}.
 By Proposition~\ref{cell-iso1} and \cite[Proposition~6.10]{RSu}, this linear isomorphism is a $\DBr$-homomorphism. Thus, $C(f, \mu',  (\nu^{o})')\cong V_{\bar\lambda}$ as right $\DBr$-modules.
 Using the universal property of Kac-modules yields  $\Hom_{U(\mfg)} (K_{\bar \l}, M_{pq}^{rt} )\cong V_{\bar\l}$
as   $\DBr$--modules (cf. the proof of Corollary~\ref{heck-kac0}).
 Now, everything is clear.
\end{proof}

In the remaining part of this section, we calculate decomposition matrices of $\DBr$. We always assume that
$p\in \mathbb Z$.  Otherwise, one can use $x_1+p_1$ instead of $x_1$  for any $p_1\in \mathbb C$ with $p-p_1\in \mathbb Z$. Since $\lambda$ is typical, we have $p-q\not\in \mathbb Z$ or
$p-q\le -m$ or $p-q\ge n$. In the first case, by \cite[Theorem~5.21]{RSu},
$\DBr$ is semisimple and hence its decomposition matrix is the identity matrix. We assume that  $p-q\le -m$. If $p-q\ge n$,  one can switch the roles between $p$ and $q$   (or by considering the dual module of $M_{pq}^{rt}$) in the following arguments.

Suppose $\lambda=(f, \mu, \nu)\in \Lambda_{2, r, t}$. Let $T_{\bar \l}$ be the indecomposable tilting module, where  $\bar \l=\l_{pq}+\tilde\l=\l_{pq}+\mu-\hat\nu$ (cf.~Definition \ref{wt-l}).
 It is the projective cover of $L_{\bar\l}$,  where $L_{\bar \l}$ is the simple $\mfg$-module with highest weight $\bar \l$.
 It is known that   $T_{\bar \l}$ has  filtrations of Kac-modules.
Let $K_{\bar \lambda^{\rm top}}$ be the unique bottom of   $T_{\bar \l}$. Then $L_{\bar \lambda}$ is the simple $\mfg$-module of $K_{\bar \lambda^{\rm top}}$. Further,  $\bar \lambda^{\rm top}$ is the dominant weight defined in Definition \ref{topp}\,(1).
  Since $M_{pq}^{rt}$ is a tilting module,  it can be decomposed into the direct sum of indecomposable tilting modules
   \begin{equation}\label{Mrtpq}
M_{pq}^{rt}=\OPL{\mu\in P^+}T_\mu^{\oplus \ell_\mu}\mbox{ \ for some \ }\ell_\mu\in\N. 
\end{equation}
In the remaining part of this paper, we denote $S$ to be the following  finite subset of $P^+$,  \begin{equation}\label{tiltings} S:=\{\mu\in P^+\,|\,\ell_\mu\ne0\}.\end{equation}
Parallel to Corollary \ref{Kl-heck-kac}, we have the following.
\begin{lemma}\label{M-tilting}
Let $\lambda=(f, \mu, \nu)\in \Lambda_{2, r, t}$ such that $(\l^{\rm top})'$ is Kleshchev, where $\l^{\rm top}$ is defined in Definition $\text{\rm\ref{topp}\,(2)}$. Then $T_{\bar\l}$ is a direct summand of $M_{pq}^{rt}$.
\end{lemma}\begin{proof}
We claim that $T_{\bar\l}$ is a direct summand in $M_{pq}^{r-f,t-f}$. If so, then
\equan{t-fff}{v_1^f\otimes T_{\bar\l}\otimes\bar v_1^f \mathfrak e^f }
is obviously a tilting submodule in $M_{pq}^{rt}$ which is isomorphic to $T_{\bar\l}$. Thus the claim implies the result. Therefore, it suffices to consider the case  $f=0$.

Denote $\bar\nu=\l_{pq}-\hat\nu$. Since $p\le q-m$, the weight diagram $D_\nu$ (cf.~Definition \ref{weight-d}) of $\bar\nu$ is obtained from that of $\l_{pq}$ in \eqref{Diagram-l}
by moving the ``\,$>$\,'' at vertex $p-i+1$ to its left side at vertex $p-i+1-\nu^{(1)}_{m-i+1}$ for each $i$ with $1\le i\le m$, and moving the ``\,$<$\,'' at vertex $q-m+j$ to its right side at vertex $q-m+j+\nu^{(2)}_{n-j+1}$ for each $j$ with $1\le j\le n$ (cf.~\eqref{hat-nu}). Thus no ``\,$\times$\,'' can be produced, i.e., $\bar\nu$ is typical. Hence $K_{\bar\nu}$ is a direct summand in $M_{pq}^{0t}$. Thus, it suffices to prove that $T_{\bar\l}$ is a direct summand in $V^{\otimes r}\otimes K_{\bar\nu}\ltimes M_{pq}^{rt}$, here $\ltimes$ means direct summand of $M_{pq}^{rt}$.
 For this, we can apply \cite[IV, Lemmas 2.4 and 2.6]{BS4}.
Note from \cite[IV, Lemma 2.4]{BS4} that the action of the functor $F_i$ on $K_{\bar\nu}$ defined in \cite[IV]{BS4}  only depends on symbols at vertices $i$ and $i+1$ of the weight diagram $D_\nu$  of $\bar\nu$
(we remark that symbols $\circ,\wedge,\vee,\times$ in \cite[IV]{BS4} are respectively symbols $<,\times,\emptyset,>$ in this paper).
Due to condition \eqref{COn11}, for any $i\in I_{pq}:=I_{pq}^+\setminus\{q-m+n\}$ such that $i$ is involved in a path in the crystal graph in \cite[IV, Lemma 2.6]{BS4}, the symbols at vertex $i$ and $i+1$ in the weight diagram $D_\nu$ of $\bar\nu$ are the same as that in the weight diagram $D_\emptyset$ of $\l_{pq}$. This shows that
$T_{\bar\l}$ is a direct summand in $V^{\otimes r}\otimes K_{\bar\nu}$ if and only if
$T_{\l_{pq}+\tilde\mu}$ is a direct summand in $V^{\otimes r}\otimes K_{\l_{pq}}$, more precisely, \cite[IV, Lemma 2.6]{BS4} implies
$$
F_{i_r}\cdots F_{i_1}K_{\bar\nu}\cong T_{\bar\l}^{\otimes 2^\ell}
\ \Longleftrightarrow\
F_{i_r}\cdots F_{i_1}K_{\l_{pq}}\cong T_{\l_{pq}+\tilde\mu}^{\otimes 2^\ell},
$$
where $\ell$ is the number of edges in the given path of the form $\emptyset\,\times\, \to\ \,<\,>$.
 Thus the result follows from Corollary \ref{Kl-heck-kac}. 
\end{proof}

 We remark that there is  a bijection between  $S$ defined in \eqref{tiltings}
and the set of pair-wise non-isomorphic simple $\DBr$-modules. See \cite[Theorem~7.5]{RSu}.
For any $\xi\in S$ as above,  parallel to Definition \ref{topp}, we define $\xi^{\rm top}$ to be the unique dominant weight such that $L_\xi$ is the simple submodule of $K_{\xi^{\rm top}}$.

\begin{proposition}~\label{tilt} For any $\xi\in S$, there is a unique $(f,\mu,\nu)\in\Lambda_{2,r,t}$ such that $\xi^{\rm top}=\lambda_{pq}+\mu-\hat\nu$ . Further,  $\FF(T_\xi)$ is isomorphic to
the projective cover  of $D^{f, \mu', (\nu^0)'}$, where  $D^{f, \mu', (\nu^0)'}$ is the simple head of  $C(f, \mu', (\nu^o)')$. \end{proposition}

\begin{proof} If $\xi\in S$, then $T_\xi$ is an indecomposable tilting module with $\ell_\xi>0$. By Theorem~\ref{level-21},  $\FF(T_\xi)$  is a direct sum of certain principle  indecomposable right $\DBr$-modules.
We claim that  $\FF(T_\xi)$ is indecomposable for any $\xi\in S$. Otherwise, $\sum_{\xi\in S} \ell_\xi$ is strictly less than the number of principal indecomposable direct summands of right  $\DBr$-module $\DBr$.
However, for each principal  indecomposable direct summand $P$ of left  $\DBr$-module $\DBr$, $P$ has to be a projective cover of irreducible left $\DBr$-module, say $D$,
which is the simple head of a left cell module, say $\Delta(\ell , \a, \b)$ for some $(\ell, \a, \b)\in \Lambda_{2, r, t}$, where  $\Delta(\ell , \a, \b)$ is defined via a weakly  cellular basis of $\DBr$.
So, there is an epimorphism from $P$ to   $\Delta(\ell, \a, \b)$. Since  $\GG:=M_{pq}^{rt}\otimes_{\DBr} ?$ is  right exact,  there is an epimorphism from $\GG(P)$ to $\GG(\Delta(\ell, \a, \b))$. If  $\GG(\Delta(\ell, \a, \b))\neq 0$, then  $\GG(P)$ is a non-zero direct summand of $M_{pq}^{rt}$.
This implies that the number of indecomposable direct summands  of left  $\DBr$-module $\DBr$ is strictly less than $\sum_{\xi\in S} \ell_\xi$. This is  a contradiction since  the number of principal indecomposable direct summands of left $\DBr$-module $\DBr$ is equal to that of right $\DBr$-module $\DBr$. So, $\FF(T_\xi)$ is indecomposable. Since $K_{\xi^{\rm top}}\hookrightarrow T_{\xi}$, we have  $\FF(T_\xi)\twoheadrightarrow \FF(K_{\xi^{\rm top}})$. By  Proposition~\ref{kac-ke}, $\FF(K_{\xi^{\rm top}})\cong C(f, \mu', (\nu^o)')$. Thus, $C(f, \mu', (\nu^o)')$ has the simple head, denoted by  $D^{f, \mu', (\nu^{o})'}$, and hence
 $\FF(T_\xi)=P^{f, \mu', (\nu^{o})'}$.  Since $\xi\in S$, by Lemma~\ref{M-tilting}, both $\mu'$ and $(\nu^{o})'$ are Kleshchev in the sense of \eqref{Kle} with respect to $-p, m-q$ and $q, p-n$.

It remains to  prove $\GG(\Delta(\ell, \a, \b^o))\neq 0$ for any $\delta:=(\ell, \a, \b)\in \Lambda_{2, r, t}$.  By Theorem~\ref{theorem-main1}, $V_{\bar\delta}$ contains
a non-zero vector $v:= v_1^{\otimes \ell} \otimes v_{\mathbf i}\otimes v_{pq}\otimes v_{\mathbf j}\otimes \bar v_1^{\otimes \ell} \mathfrak e^f  w_{\a, \b} \mathfrak y_{\a'} \bar {\mathfrak y}_{(\b^{o})'}$,
where $\mathbf i$ and $\mathbf j$ are defined as in Definition~\ref{mixh}.
So, it is enough to show $v\in  \GG(\Delta(\ell, \a, \b^o))$, where $\Delta(\ell, \a, \b^o)$ is defined via a suitable weakly cellular basis of $\DBr$. We use cellular bases of $\mathscr H_{2,r-f}$ and $\mathscr H_{2, t-f}$
in Lemma~\ref{cell-h2}\,(1)\,(3) to construct a weakly cellular basis of $\DBr$, which is similar to that in Theorem~\ref{cellular-1}. Let $\Delta(\ell, \a, \b^o)$ be the corresponding left cell module with respect to
$(\ell, \a, \b^{o})\in \Lambda_{2, r, t}$. By arguments similar to those for the proof of Proposition ~\ref{classcell2}, one can verify
$$\Delta(\ell, \a, \b^o)\cong\DBr \mathfrak e^f \mathfrak x_\a \bar{\mathfrak x}_{\b^o} w_{\a, \b} \mathfrak y_{\a'}\bar{\mathfrak y}_{(\b^{o})'} \pmod{ \mathscr \B_{2, r, t}^{\ell+1}}.  $$
Let $M=\tilde v \DBr$ be the cyclic $\DBr$-module generated by $\tilde v :=
v_1^{\otimes \ell} \otimes v_{\mathbf i}\otimes v_{pq}\otimes v_{\mathbf j}\otimes \bar v_1^{\otimes \ell}$. Then  $M\otimes_{\DBr}\Delta(\ell, \a, \b^o) $ is a subspace of $\GG(\Delta(\ell, \a, \b^o))$.
Since   $\mathscr \B_{2, r, t}^{\ell+1}$ acts on $M$ trivially, there is  a $\mathbb C$-linear map $\phi: M\otimes_{\DBr}\Delta(\ell, \a, \b^o)\rightarrow M$ such that $\phi(m\otimes \bar h)=m h$ for any
$\bar h\in  \DBr \mathfrak e^f \mathfrak x_\a \bar{\mathfrak x}_{\b^o} w_{\a, \b} \mathfrak y_{\a'}\bar{\mathfrak y}_{(\b^{o})'} \pmod{ \mathscr \B_{2, r, t}^{\ell+1}}$. Since $\lambda_{pq}$ is typical and the ground filed is $\mathbb C$,
up to a non-zero scalar, we have $v=\phi(\tilde v\otimes \bar h)$ , where $h\equiv \mathfrak  e^f  w_{\a, \b} \mathfrak y_{\a'} \bar {\mathfrak y}_{(\b^{o})'}   \pmod{ \mathscr \B_{2, r, t}^{\ell+1}}$.
Thus,  $\GG(\Delta(\ell, \a, \b^o))\neq 0$.
\end{proof}

\begin{remark} Proposition~\ref{tilt} implies that $C(f, \mu, \nu)$ has the simple head if $\mu$ and $\nu$ are Kleshchev bipartitions with respect to $-p, m-q$ and $q, p-n$ in the sense of \eqref{Kle}. Further, all non-isomorphic  simple $\DBr$-modules can be realized in this way.\end{remark}

 \begin{proposition}\label{fsimple} Suppose  $\xi\in P^+$. Then  $\FF(L_{{\xi}})=0$ if $\xi\not\in S$ $($cf.~\eqref{tiltings}$)$ and  $\FF(L_{{\xi}})\cong D^{f,\mu',(\nu^{o})'}$ if  $\xi\in S$,
  where  $\xi^{\rm top}=\lambda_{pq}+\mu-\hat\nu$ with $(f, \mu, \nu)\in \Lambda_{2, r, t}$.\end{proposition}
\begin{proof} 
By \eqref{Mrtpq},  $\FF(L_\xi)=\oplus_{\zeta\in S}{\rm Hom}_\mfg(L_\xi,T_\zeta^{\otimes
\ell_\zeta})$. Suppose $0\ne f\in{\rm Hom}_{U(\mfg)}(L_\xi,T_\zeta^{\oplus \ell_\zeta})$. Then
$L_\xi\cong f(L_\xi)$ is a simple submodule of $T_\zeta^{\oplus \ell_\zeta}$.
Since $T_\zeta$ has the unique simple submodule $L_\zeta$,
$\FF(L_\xi)=0$ if $\xi\notin S$.  If $\xi\in S$, then  \equa{F-L=}{\FF(L_\xi)={\rm Hom}_{U(\mfg)}(L_\xi,T_\xi^{\oplus \ell_\xi}),} which is obviously  $\ell_\xi$-dimensional.
Let $v_{\xi}^1,...,v_\xi^{\ell_\xi}\in T_\xi^{\oplus\ell_\xi}$ be the generators of the tilting module
$T_\xi^{\oplus\ell_\xi}$ (then $v_{\xi}^1,...,v_\xi^{\ell_\xi}$ span the generating space, denoted $\textbf{\textit{V}}$, of $T_\xi^{\oplus\ell_\xi}$), and $v'^1_{\xi},...,v'^{\ell_\xi}_{\xi}\in L_\xi^{\oplus\ell_\xi}
$, the corresponding generators of the submodule $ L_\xi^{\oplus\ell_\xi}$ of $ T_\xi^{\oplus\ell_\xi}$. Thus, there exists a unique $u\in U(\mfg)$ such that \equa{v-prime}{v'^i_{\xi}=uv_\xi^i\mbox{ \ for \ }i=1,...,\ell_\xi.} Let $\tilde v_\xi\in L_\xi$ be the generator of the simple module $L_\xi$. As in the proof of Corollary \ref{heck-kac}, we can define  $f^i:L_\xi\to T_\xi^{\oplus\ell_\xi}$ to be the $U(\mfg)$-homomorphism sending $\tilde v_\xi$ to $v'^i_\xi$ for $i=1,...,\ell_\xi$. Then $(f^1,...,f^{\ell_\xi})$  is obviously a basis of $\FF(L_\xi)$ (cf.~\eqref{F-L=}).

For any $A\in M_{\ell_\xi}$ (the algebra of $\ell_\xi\times\ell_\xi$ complex matrices),
we can define an element $\phi_A\in {\rm End}_{U(\mfg)} (M_{pq}^{rt})=\DBr$ as follows:
$\phi_A\Big|_{T_\zeta^{\oplus\ell_\zeta}}=0$ if $\zeta\ne\xi$ and
\equa{phi-A-xi}{\phi_A\Big|_{T_\xi^{\oplus\ell_\xi}}: (v_\xi^1,...,v_\xi^{\ell_\xi})\mapsto(v_\xi^1,...,v_\xi^{\ell_\xi})A
\mbox{ \ (regarded as vector-matrix multiplication)},}
i.e., the transition matrix of the action of $\phi_A\Big|_{T_\xi^{\oplus\ell_\xi}}$ on the generating space $\textbf{\textit{V}}$ of $T_\xi^{\oplus\ell_\xi}$ \vspace*{-5pt}under the basis $(v_\xi^1,...,v_\xi^{\ell_\xi})$  is $A$. By the universal property of projective modules, this uniquely defines an element $\phi_A\in\DBr$. Thus we have the embedding $\phi:M_{\ell_\xi}\to\DBr$ sending $A$ to $\phi_A$.
Write $A$ as $A=(a_{ij})_{i,j=1}^{\ell_\xi}$. Then by \eqref{phi-A-xi} and
 definition of the right action of $\DBr$ on $M_{pq}^{rt}$, we have
\begin{eqnarray}\label{phi-A-action}
\!\!\!\!\!\!&\!\!\!\!\!\!\!\!\!\!\!\!\!\!\!&
f^i(\tilde v_\xi)\phi_A=v'^i_\xi\phi_A=(uv_\xi^i)\phi_A=u(v_\xi^i\phi_A)=u\mbox{$\sum\limits_{j=1}^{\ell_\xi}$}
a_{ji}v_\xi^j
=\mbox{$\sum\limits_{j=1}^{\ell_\xi}$}a_{ji}v'^j_\xi=\Big(\mbox{$\sum\limits_{j=1}^{\ell_\xi}$}a_{ji}f^j\Big)(\tilde v_\xi),
\end{eqnarray}
i.e., the transition matrix of the action of $\phi_A$ on
$\FF(L_\xi)$ under the basis $(f^1,...,f^{\ell_\xi})$ is $A$.
Thus $\phi(M_{\ell_\xi})$ acts transitively on
 the $\ell_\xi$-dimensional space $\FF(L_\xi)$ and hence  $\FF(L_\xi)$ is a simple $\DBr$-module.
 Finally, since $L_\xi\hookrightarrow K_{\xi^{\rm top}}$, we have  $\FF(K_{\xi^{\rm top}})\twoheadrightarrow\FF(L_\xi)$. Note that   $D^{f,\mu',(\nu^{o})'}$ is the simple head of
 $\FF(K_{\xi^{\rm top}})$. Thus,  $\FF(L_\xi)\cong D^{f,\mu',(\nu^{o})'}$.
\end{proof}

\begin{theorem} Suppose $(f, \alpha, \beta)\in \Lambda_{2, r, t}$ such that  there is a  $\l\in S$ $($cf. \eqref{tiltings}$)$
satisfying $\l^{\rm top}=\l_{pq}+\a-\hat \b$. If  $\mu:=(\ell, \gamma, \delta)\in \Lambda_{2, r, t}$, then
$[C(\ell, \gamma', (\delta^{o})'): D^{f, \a', (\b^{o})'} ]=(T_{\l}: K_{\bar \mu})$. \end{theorem}
\begin{proof} The result follows from Lemma~\ref{exact}, Propositions~\ref{kac-ke} and \ref{fsimple}, together with the BGG reciprocity formula for $\mfg$.
\end{proof}


\small\begin{thebibliography}{DWH99}

\bibitem{AM:simples}
{\scshape S.~Ariki {\normalfont \smfandname} A.~Mathas}, {\og {The number of
  simple modules of the Hecke algebras of type $G(r,1,n)$}\fg}, \emph{Math. Z.}
  \textbf{233} (2000), 601--623.

\bibitem{Ariki:class}
{\scshape S.~Ariki {\normalfont \smfandname} A.~Mathas}, {\og On the classification of simple modules for cyclotomic {Hecke
  algebras of type $G(m,1,n)$ and Kleshchev} multipartitions\fg}, \emph{Osaka
  J.~Math.} \textbf{38} (2001), 827--837.

\bibitem{AMR}{\scshape S.~Ariki, A.~Mathas {\normalfont \smfandname}  H. Rui},
{\og Cyclotomic Nazarov-Wenzl algebras\fg}, \emph{Nagoya Math.
J.}, Special issue in honor of Prof. G. Lusztig's sixty birthday,
  \textbf{182} (2006), 47--134.

\bibitem{B}  {\scshape J.  Brundan}, {\og Kazhdan--Lusztig polynomials and character formulae for the Lie superalgebra
$\frak{gl}_{m|n}$\fg}, \emph{J. Amer. Math. Soc.} \textbf{16} (2002), 185--231.

\bibitem{BS4} {\scshape J. Brundan  {\normalfont \smfandname} C. Stroppel}, {\og Highest weight categories arising from Khovanov's diagram algebra
I, II, III, IV\fg}, \emph{Moscow Math. J.} \textbf {11}, (2011), 685--722; \emph{Transform. Groups} \textbf{15}, (2010),  1--45; \emph {Represent.
Theory} \textbf{15}, (2011),  170--243; \emph{J. Eur. Math. Soc.} \textbf{14},  (2012), 373--419.

\bibitem {BS} {\scshape J. Brundan   {\normalfont \smfandname} C. Stroppel}, {\og  Gradings on walled Brauer algebras and Khovanov's arc algebra}, \emph{ Adv. Math.} \textbf {231} (2012), no. 2, 709--773.

\bibitem{DJ2} {\scshape R. Dipper and G. G. James} {\og Representations of Hecke algebras  of type $B_n$ \fg}, \emph{ J. Algebra } \textbf{146} (1992), 454--481.

\bibitem {DJM} {\scshape R. Dipper and G.  James and E. Murphy} {\og Hecke algebras of type $B_n$    at roots of unity \fg}, \emph{ Proc. London Math. Soc. } (3) \textbf {70}  (1995),  no. 3, 505--528.

\bibitem{DM:Morita}
{\scshape R.~Dipper {\normalfont \smfandname} A.~Mathas}, {\og Morita
  equivalences of {Ariki--Koike} algebras\fg}, \emph{Math. Z.} \textbf{240}
  (2002), 579--610.


\bibitem{Dri} {\scshape V.G. Drinfeld}, {\og  Degenerate affine Hecke algebras and Yangians\fg}, \emph{ Func. Anal. Appl.}, \textbf{20} (1986), 62--64.


\bibitem{GG}
{\scshape F.~ Goodman {\normalfont \smfandname} J. Graber}, {\og Cellularity and the Jones basic construction
\fg}, \emph{Adv. in  Appl. Math.} \textbf{46} (2011),  312--362

\bibitem{GL}
{\scshape J.~J. Graham {\normalfont \smfandname} G.~I. Lehrer}, {\og
Cellular  algebras\fg}, \emph{Invent. Math.} \textbf{123} (1996), 1--34.

\bibitem{GS} {\scshape C. Groson {\normalfont \smfandname} V. Serganova}, {\og Cohomology of generalized supergrassmannians and character
formulae for basic classical Lie superalgebras\fg}, \emph{Proc. London Math. Soc.} \textbf{101} (2010), 852--892.

\bibitem{Kac77} {\scshape V.G.~Kac}, {\og Lie superalgebras\fg}, \emph{Adv.~Math.} \textbf{26} (1977), 8--96.

\bibitem{K} {\scshape A. Kleshchev}, {\og Linear and projective representations of symmetric groups\fg}, { Cambridge Tracts in Mathematics, 163. Cambridge University Press, Cambridge}, 2005.

\bibitem{Koi} {\scshape K. Koike}, {\og On the decomposition of tensor products of the representations of classical groups: By means of
universal characters\fg}, \emph{Adv. Math.} \textbf {74} (1989) 57--86.



\bibitem{N} {\scshape P. Nikitin}, {\og  The centralizer algebra of the diagonal action of the group GLn(C) in
a mixed tensor space\fg}, \emph{J. Math. Sci.} \textbf {141} (2007), 1479--1493.

\bibitem{RSong}{\scshape H. Rui {\normalfont \smfandname} L. Song}, {\og Decomposition numbers of
quantized walled Brauer algebras\fg}, preprint, 2011.

\bibitem{RSu}{\scshape H. Rui {\normalfont \smfandname} Y. Su}, {\og  Affine walled Brauer algebras and super Schur-Weyl duality\fg},
 arXiv:1305.0450.

\bibitem {Sa} {\scshape A. Sartori},{\og The degenerate affine walled Brauer algebra\fg}, arXiv: 1305.2347


\bibitem{SHK} {\scshape Y.~Su, J.W.B.~Hughes {\normalfont \smfandname}
R.C.~King}, {\og Primitive vectors in the Kac-module of the Lie
superalgebra $sl(m|n)$\fg}, \emph{J.~Math.~Phys.} \textbf{41} (2000), 5044--5087.

\bibitem{SZ3}	{\scshape Y.~Su {\normalfont \smfandname}  R.B. Zhang}, {\og  Generalised Jantzen filtration of Lie superalgebras I\fg}, \emph{J. Eur. Math. Soc.} \textbf{14} (2012), 1103--1133.

\bibitem {Tur} {\scshape V.~Turaev}, {\og  Operator invariants of tangles and R-matrices\fg}, \emph{Izv. Akad. Nauk SSSR Ser. Math.} \textbf{53} (1989) 1073--1107 (in Russian).

\bibitem{vZ} {\scshape J.~van der Jeugt {\normalfont \smfandname}  R.B. Zhang}, {\og  Characters and composition factor multiplicities for the Lie
superalgebra $gl(m|n)$\fg}, \emph{Lett. Math. Phys.} \textbf{47} (1999), 49--61.

\bibitem{Vazi} {\scshape M. Vazirani}, {\og Parameterizing Hecke algebra modules: Bernstein-Zelevinsky multisegments, Kleshchev multipartitions,
and crystal graphs\fg }, \emph{Transformation Groups} \textbf{7} (2002), no.~3, 267--303\vspace*{-7pt}.
\end{thebibliography}
\end{document}